\documentclass[a4paper,10pt,leqno,twoside]{tobart}
\usepackage[pdftex]{graphicx} 
\usepackage[english]{babel} \usepackage{inputenc, amsmath, amssymb, latexsym, rotating, fancyhdr, amsthm, pifont, mathabx}
\usepackage{color}

\newcommand{\inter}{\operatorname{int}}
\newcommand{\til}{\widetilde}
\newcommand{\sm}{\setminus}
\newcommand{\ol}{\overline}
\newcommand{\bd}{\partial}
\newcommand{\ie}{i.e.\ }
\DeclareMathOperator{\diff}{\rm{Diff}}
\newcommand{\id}{\mathrm{Id}}

\newcommand{\foot}{\footnote}

\newcommand{\ssq}{\ensuremath{\subseteq}}
\newcommand{\smin}{\ensuremath{\setminus}}
\newcommand{\eps}{\ensuremath{\varepsilon}}
\newcommand{\wh}{\ensuremath{\widehat}}

\newcommand{\dist}{\ensuremath{\mathop{\rm dist}}}

\newcommand{\diam}{\ensuremath{\mathrm{diam}}}

\newcommand{\kreis}{\ensuremath{\mathbb{T}^{1}}}

\newcommand{\alphlist}{\begin{list}{(\alph{enumi})}{\usecounter{enumi}\setlength{\parsep}{2pt}
      \setlength{\itemsep}{1pt} \setlength{\topsep}{5pt}
      \setlength{\partopsep}{3pt}}}
\newcommand{\arablist}{\begin{list}{(\arabic{enumi})}{\usecounter{enumi}\setlength{\parsep}{2pt}
          \setlength{\itemsep}{1pt} \setlength{\topsep}{5pt}
          \setlength{\partopsep}{3pt}}}
\newcommand{\romanlist}{\begin{list}{(\roman{enumi})}{\usecounter{enumi}\setlength{\parsep}{2pt}
              \setlength{\itemsep}{1pt} \setlength{\topsep}{5pt}
              \setlength{\partopsep}{3pt}}}
\newcommand{\Romanlist}{\begin{list}{(\Roman{enumi})}{\usecounter{enumi}\setlength{\parsep}{2pt}
              \setlength{\itemsep}{1pt} \setlength{\topsep}{5pt}
              \setlength{\partopsep}{3pt}}}
\newcommand{\bulletlist}{\begin{list}{$\bullet$}{\setlength{\parsep}{2pt}
                \setlength{\itemsep}{1pt} \setlength{\topsep}{5pt}
                \setlength{\partopsep}{3pt}\setlength{\leftmargin}{15pt}}} 
\newcommand{\Alphlist}{\begin{list}{(\Alph{enumi})}{\usecounter{enumi}\setlength{\parsep}{2pt}
      \setlength{\itemsep}{1pt} \setlength{\topsep}{5pt}
      \setlength{\partopsep}{3pt}}}
 \newcommand{\listend}{\end{list}}

\newcommand{\T}{\ensuremath{\mathbb{T}}}

\newcommand{\N}{\ensuremath{\mathbb{N}}} 
\newcommand{\R}{\ensuremath{\mathbb{R}}}
\newcommand{\Z}{\ensuremath{\mathbb{Z}}}
\newcommand{\Q}{\ensuremath{\mathbb{Q}}}

\newcommand{\A}{\ensuremath{\mathbb{A}}}

\newcommand{\cA}{\mathcal{A}}
\newcommand{\cB}{\mathcal{B}}
\newcommand{\cC}{\mathcal{C}}

\newcommand{\cF}{\mathcal{F}}

\newcommand{\cH}{\mathcal{H}}

\newcommand{\cL}{\mathcal{L}}

\newcommand{\cR}{\mathcal{R}}

\newcommand{\cU}{\mathcal{U}}

\newcommand{\cX}{\mathcal{X}}
\newcommand{\cY}{\mathcal{Y}}

\newcommand{\iLim}{\ensuremath{\lim_{i\rightarrow\infty}}}

\usepackage{enumitem}

\newtheoremstyle{tobthm}{3pt}{3pt}{\itshape}{0pt}{\bfseries}{.}{0.5eM}{}
\theoremstyle{tobthm}

\newtheorem{definition}{Definition}[section]
\newtheorem{thm}[definition]{Theorem}

\newtheorem{lem}[definition]{Lemma}
\newtheorem{cor}[definition]{Corollary}  
\newtheorem{prop}[definition]{Proposition}

 \newtheorem{claim}[definition]{Claim}

\newtheoremstyle{tobrem}{3pt}{3pt}{\normalfont}{0pt}{\bfseries}{.}{0.5em}{}
\theoremstyle{tobrem}

\newtheorem{rem}[definition]{Remark}

\numberwithin{equation}{section}
\numberwithin{figure}{section}

\title{\large\textsc{Poincar\'e theory for decomposable cofrontiers}}

\author{T.~J\"ager\thanks{Department of Mathematics, TU Dresden, Germany. Email:
    {\tt Tobias.Oertel-Jaeger@tu-dresden.de}}\ \ and
  A.~Koropecki\thanks{Universidade Federal Fluminense, Niter\'oi, Brasil. Email: {\tt ak@id.uff.br}}}
\date{}
\newcommand{\filled}{\ensuremath{\textrm{fill}}}

\begin{document}

\setlength{\abovedisplayskip}{1.0ex}
\setlength{\abovedisplayshortskip}{0.8ex}

\setlength{\belowdisplayskip}{1.0ex}
\setlength{\belowdisplayshortskip}{0.8ex}
\maketitle

\abstract{We extend Poincar\'e's theory of orientation-preserving
  homeomorphisms from the circle to circloids with decomposable
  boundary. As special cases, this includes both decomposable
  cofrontiers and decomposable cobasin boundaries. More precisely, we
  show that if the rotation number on an invariant circloid $A$ of a
  surface homeomorphism is irrational and the boundary of $A$ is
  decomposable, then the dynamics are monotonically semiconjugate to
  the respective irrational rotation.  This complements classical
  results by Barge and Gillette on the equivalence between rational
  rotation numbers and the existence of periodic orbits and yields a
  direct analogue to the Poincar\'e Classification Theorem for circle
  homeomorphisms. Moreover, we show that the semiconjugacy can be
  obtained as the composition of a monotone circle map with a
  `universal factor map', only depending on the topological structure
  of the circloid. This implies, in particular, that the monotone
  semiconjugacy is unique up to post-composition with a rotation.

  If, in addition, $A$ is a minimal set, then the semiconjugacy is
  almost one-to-one if and only if there exists a biaccessible point.
  In this case, the dynamics on $A$ are almost automorphic.
  Conversely, we use the Anosov-Katok method to build a
  $C^\infty$-example where all fibres of the semiconjugacy are
  non-trivial.}

\noindent
\section{Introduction}

Given an orientation-preserving circle homeomorphism $\varphi$, the
Poincar\'e Classification Theorem states that the rotation number
$\rho(\varphi)$ is well-defined and determines the qualitative
dynamical behaviour of $\varphi$, in the sense that $\rho(\varphi)$ is
rational if and only if $\varphi$ has a periodic orbit and irrational
if and only if $\varphi$ is monotonically semiconjugate to the
respective irrational rotation. This result provides the basis for a
rather complete understanding of invertible dynamics on the circle. At
the same time the study of cofrontiers, circloids and other classes of
circle-like continua, like basin boundaries, co-basin boundaries or
pseudocircles, has a long history in plane topology and continuum
theory, going back to Kuratowski, Cartwright, Littlewood, Bing and
others
\cite{kura-frontier,cartwright-littlewood,Bing1951Pseudocircle}.
Recently the topic has gained further momentum, since invariant
circloids play a crucial role in surface dynamics. It is therefore a
natural question to ask whether an analogue to Poincar\'e's classical
result holds for these more general continua.  In this article,
we extend the Poincar\'e Classification to circloids with decomposable
boundary.\foot{We recall that a continuum is called {\em decomposable}
  if it can be written as the union of two non-empty proper
  subcontinua.}  In order to define this notion, we let $\kreis=\R/\Z$
and $\A=\kreis\times \R$ and call a continuum $A\ssq\A$ an {\em
  essential annular continuum} if $\A\smin A$ consists of exactly two
connected components, both of which are homeomorphic to $\A$. Further,
$A$ is called an essential circloid if it does not contain any
strictly smaller essential annular continuum as a subset, and a cofrontier if it is a circloid with empty interior.
We refer to Section~\ref{Preliminaries} for further explanations and details.

In the context of the dynamics of surface homeomorphisms, circloids
may appear in various situations. For instance, they separate adjacent
invariant topological disks or annular domains
\cite{BargeGillette1991CofrontierRotationSets,brechner-et-al,WalkerPrimeEnd,Kennedy1994Pseudocircles,franks/lecalvez:2003},
and any periodic point free continuum of a non-wandering surface
homeomorphism is an annular continuum \cite{koropecki2010aperiodic}
that can further be decomposed into a dense union of invariant
circloids and transitive annuli \cite{jaeger:2010a}.  On the
two-torus, the existence of invariant circloids can often be deduced
from information on the rotation set
\cite{jaeger:2009c,Davalos2013SublinearDiffusion,GuelmanKoropeckiTal2012Annularity},
and invariant ``foliations'' consisting of circloids play an important
role for the problem of linearisation \cite{jaeger:2009b}. The
respective results in topological dynamics have further applications
in the theory of $C^r$-generic diffeomorphisms
\cite{franks/lecalvez:2003,KoropeckiNassiri2010GenericTransitivity,KoroLeCalvezNassIrrPE}.
It is thus of vital interest to understand the interplay between the
topological structure and the possible dynamical behaviour on such
continua. However, while the relation between rational rotation
numbers and periodic orbits is quite well-understood
\cite{cartwright-littlewood,BargeGillette1991CofrontierRotationSets,KoroLeCalvezNassIrrPE},
the more intricate question of irrational rotation factors has been
left completely open so far.

The problem is complicated by the fact that the rotation number on
invariant circloids is not necessarily unique. Non-degenerate rotation
intervals have been shown to occur on the Birkhoff attractor
\cite{LeCalvez1988BirkhoffAttractor} and, more recently, the
pseudocircle \cite{BoronskiOprocha2014InverseLimitPseudocircle}. Such
examples can be excluded by adding a mild recurrence assumption
\cite{KoroLeCalvezNassIrrPE}, but even in the case of a unique
rotation number a semiconjugacy does not have to exist. This was shown
by Handel \cite{handel:1982} and Herman \cite{herman:1986}, who
realised the pseudocircle as a minimal set of a smooth surface
diffeomorphism. In these examples, the rotation number is irrational,
but the dynamics are not semiconjugate to the corresponding rotation.
While the pseudocircle is the paradigm example of a circle-like
continuum with highly intricate topological structure, a modification
of the construction can be used to produce a variety of more `regular'
indecomposable continua with the same behaviour. Hence, {\em
  decomposability} of the circloid presents itself as the obvious
minimal requirement for a possible analogue to the Poincar\'e
classification. As the following result shows, it turns out to
be sufficient as well. Recall that a monotone map is one with connected fibers. 

\begin{thm} \label{t.poincare} Suppose $\varphi\colon \A \to \A$ is a
  homeomorphism homotopic to the identity with an essential
  $\varphi$-invariant circloid $A$ with decomposable boundary. Then
  every point of $A$ has a well-defined rotation number $\rho\in
  \kreis$ which is independent of the point, and
\begin{itemize}
\item $\rho$ is rational if and only if there is a periodic point in
  $A$.
\item $\rho$ is irrational if and only if $\varphi|_A$ is
  monotonically semiconjugate to the corresponding irrational rotation
  by $\rho$ on $\kreis$.
\end{itemize}
\end{thm}

Barge and Gillette showed that the rotation number on a decomposable
cofrontier is always unique, and it is rational if and only if there
exists a periodic orbit in the cofrontier
\cite{BargeGillette1991CofrontierRotationSets}.
Theorem~\ref{t.poincare} complements these results to give a direct
analogue to the Poincar\'e Classification Theorem for decomposable
cofrontiers and, more generally, to circloids with decomposable
boundary. Thus, dynamics with irrational rotation number on an
invariant circloid with decomposable boundary are `linearisable'.  

As it further turns out, to a great extent this linearisation does not
depend on the dynamics. More precisely, there exists a `universal
factor map' which maps the circloid to a topological circle and
semiconjugates the dynamics of {\em any} homeomorphism preserving the
circloid to that of a circle homeomorphism.

\begin{thm} \label{t.universal_factor_map} Suppose that $A\ssq \A$ is
  an essential circloid with decomposable boundary and there exists a
  self-homeomorphism of $\A$ leaving $A$ invariant without periodic
  points in $A$. Then there exists a continuous and onto map
  $\Pi:\A\to\A$ with the following properties.  \romanlist
  \item $\Pi$ is monotone and homotopic to the identity;
  \item $\Pi$ sends $A$ to $\T=\kreis\times\{0\}$;
  \item $\Pi$ is injective on $\A\smin A$;
  \item for any homeomorphism $\varphi\colon\A\to \A$ such that $\varphi(A)=A$ there exists a homeomorphism $\tilde \varphi\colon
    \A\to \A$ such that $\tilde \varphi(\T)=\T$ and $\Pi\circ \varphi
    = \tilde\varphi\circ\Pi$.
  \item If $h\colon A\to \kreis$ is any monotone surjection, then
    there exists a monotone map $\til{h}\colon \T \to \kreis$ such
    that $h=\tilde h\circ \Pi$. In particular, if $\varphi\colon \A\to \A$ is a
    homeomorphism leaving $A$ invariant and $h$ semiconjugates
    $\varphi|_{A}$ to an irrational rotation $R_\rho$, then $\tilde h$
    semiconjugates $\tilde \varphi|_{\T}$ to $R_\rho$ (where $\tilde \varphi$ is as in the previous item).  \listend
\end{thm}
We give a short self-contained proof in
Section~\ref{UniversalFactorMap}. It turns out, however, that the
family of subcontinua of $A$ given by the fibres of $\Pi$ coincides
with a decomposition of the circloid constructed already by
Kuratowski, in a purely topological context
\cite{kura-frontier}. We discuss this in
Section~\ref{Kuratowski}.

It is well-known that the semiconjugacies in the Poincar\'e
Classification Theorem are unique up to post-composition by a
rotation. As an immediate consequence of
Theorem~\ref{t.universal_factor_map}, we obtain the same statement for
decomposable circloids.
  \begin{cor} \label{c.uniqueness} The semiconjugacy in Theorem~\ref{t.poincare}
    is unique up to post-composition by a rotation.
\end{cor}

It is worth mentioning that we do not know if an indecomposable cofrontier which is invariant by a homeomorphism may be semi-conjugate to an irrational rotation (however, if such an example exists the semiconjugacy cannot be monotone).

As should be expected, additional information on the topological
structure of the circloid yields further information on the dynamics.
We concentrate on the relation between the existence of biaccessible
points and almost automorphic dynamics, whose study is a classical
topic in abstract topological dynamics
\cite{veech1965almost,ellis1969lectures,auslander1988minimal}. A
homeomorphism $\varphi\colon X\to X$ is almost automorphic if it is
semiconjugate to some almost periodic homeomorphism of a space $Y$ in
a way that the set of points of $Y$ with a unique preimage under the
semiconjugation is dense in $Y$.  The following statement shows how
sets of this type appear in surfaces. A point of an essential circloid
$A\subset \A$ is called \emph{biaccessible} if it is the unique
intersection point of some arc $\sigma$ with $A$ such that $\sigma$
intersects both components of $\A\sm A$. An essential cobasin boundary $B$ is the boundary of an essential circloid, $B=\bd A$, and if $A$ has a biaccessible point $x$ belonging to $B$, we  also say that $x$ is a biaccessible
point of $B$.

\begin{thm} \label{t.almost_automorphy} If $B$ is an essential cobasin
  boundary in $\A$ invariant by a homeomorphism $\varphi\colon \A\to
  \A$ without periodic points and there is a biaccessible recurrent point in $B$, then
  $\varphi|_B$ is almost automorphic.
\end{thm}

\begin{cor} \label{c.almost_automorphy} Let $\varphi\colon \A\to \A$ be
  a homeomorphism and $X$ is an essential $\varphi$-invariant
  continuum such that $\varphi|_X$ is minimal. If $X$ has a
  biaccessible point, then $X$ is a decomposable cofrontier and
  $\varphi|_X$ is almost automorphic.
\end{cor}

We note that if a planar homeomorphism is both almost periodic and minimal on an invariant continuum $X$, then $X$ is a simple closed curve \cite{brechner-et-al}.

In \cite{MMO}, examples are constructed, using the Anosov-Katok method, of decomposable cofrontiers arising as the boundary of Siegel disks, which admit a monotone surjection onto $\T^1$ for which there is a subset of point inverses homeomorphic to the product of a cantor set and an interval (in particular, uncountably many fibers are nontrivial). 
All examples of minimal decomposable cofrontiers that we found in the literature are, to our knowledge, almost automorphic; see for instance \cite{WalkerPrimeEnd,herman:1983}.
We close with  a construction that is also based on the Anosov-Katok method \cite{anosov/katok:1970} and demonstrates that this is not necessarily the case: \emph{all} fibres may be non-trivial.

\begin{thm} \label{t.example} There exists a $C^\infty$ diffeomorphism
  $\varphi\colon \A\to \A$ leaving invariant a decomposable cofrontier
  $A$ such that \romanlist
  \item the rotation number on $A$ is irrational; 
  \item the dynamics on $A$ are minimal;
  \item all the fibres of points of $\T^1$ of the semiconjugacy given by Theorem \ref{t.universal_factor_map} are non-trivial continua (\ie not a single point). \listend
\end{thm}
In fact, the fibres of points of $\T^1$ have a diameter uniformly bounded below by a positive constant (see Claim \ref{claim:ak-2}) and, although we do not give a formal proof, it can be seen from the
construction that all these fibres can be given a rich topological structure,
reminiscent of the Knaster Buckethandle continuum.  

We note that due to the nature of the construction, the rotation number of the example from Theorem \ref{t.example} is  Liouvillean. We do not know whether such an example exists with a Diophantine rotation number. Similar questions in the context of indecomposable cofrontiers have been raised by in \cite{brechner-et-al} (see also \cite{turpin}).

\bigskip

\noindent{\bf Acknowledgements.} The authors are grateful to the anonymous referee for the suggestions that helped improve this paper. This work was initiated during the
conference {\em `Surfaces in S\~ao Paulo'}, held in S\~ao Sebasti\~ao,
7-11 April 2014. TJ would like to thank the organisers and
participants for creating this unique opportunity. 
This cooperation was
supported by the Brasilian-European exchange program BREUDS (EU Marie
Curie Action, IRSES Scheme). TJ acknowledges support of the German
Research Council (Emmy Noether Grant Ja 1721/2-1). AK acknowledges
support of CNPq-Brasil and FAPERJ-Brasil.

\section{Notation and preliminaries} 

\label{Preliminaries} We denote
by $\A = \kreis\times \R$ the open annulus, and $\pi\colon \R^2\to \A$
its universal covering map, where $T:(x,y)\mapsto (x+1, y)$ is a
generator of the group of covering transformations.

A subset $A$ of the open annulus $\A=\kreis\times \R$ is called an
{\em essential annular continuum} if it is compact and connected and
its complement $\A\smin A$ consists of exactly two connected
components, both of which are unbounded. Note that in this situation
one of the components is unbounded above and bounded below, whereas
the other is bounded above and unbounded below, and both of them are
homeomorphic to $\A$. Moreover, $A$ is the decreasing intersection of
a sequence of closed annuli.  We call $C\ssq \cU$ an {\em essential
  circloid} if it is a minimal element with respect to inclusion
amongst essential annular continua. An essential circloid with empty
interior is called an {\em essential cofrontier}. The boundary of an
essential circloid is called an {\em essential cobasin boundary}. It
is the intersection of the boundaries of the two complementary
components of the circloid and a minimal element with respect to
inclusion amongst essential continua.  A subset $A$ of a surface $S$
is called {\em annular continuum (circloid/cofrontier/cobasin
  boundary)} if it has a neighbourhood $\cU$ homeomorphic to $\A$ such
that $A$ as a subset of $\cU$ is an essential annular continuum
(essential circloid/essential cofrontier/essential cobasin boundary)
in the above sense. Note that thus an annular continuum in $\A$ may be
non-essential, in which case it is contained in a closed topological
disk.  From now on, given any annular continuum, circloid, cobasin
boundary or cofrontier, we always identify its annular neighbourhood
$\cU$ with $\A$ and assume implicitly that the objects are essential
in $\A$.

A closed subset $\cA\ssq\R^2$ is called {\em horizontal}, if there
exists $M>0$ with $\cA\ssq\R\times[-M,M]$, and {\em horizontally
  separating} if $\R\times(-\infty,-M)$ and $\R\times(M,\infty)$ are
contained in different connected components of $\R^2\smin A$. It is
called a {\em horizontal strip} if it separates the plane into exactly
two connected components, one of them unbounded above and the other
unbounded below. A horizontal strip is called {\em minimal} if it does
not strictly contain a smaller horizontal strip. In this case, its
boundary is called a {\em horizontal coplane boundary} and equals the
intersection of the boundaries of the two complementary domains of the
strip. A horizontal coplane boundary is minimal amongst horizontally
separating sets.  If $A$ is an essential continuum, we call the set
$\cA=\pi^{-1}(A)$ its lift. We state the next observation as a lemma,
since it will be used repeatedly. Its proof is straightforward and
left to the reader.
\begin{lem}
  \label{l.minimal_strips} 
  The lift of an essential continuum $A$ is a minimal strip if and
  only if $A$ is a circloid, and the lift of an essential continuum
  $B$ is a coplane boundary if and only if $B$ is a cobasin boundary.
\end{lem}

Let $\varphi\colon \A\to \A$ be a homeomorphism homotopic to the identity. Any such map lifts to a homeomorphism $\Phi\colon \R^2\to \R^2$ which commutes with the deck
transformation $T:(x,y)\mapsto (x+1,y)$. If $A$ is a compact invariant
subset of $\varphi$, the {\em rotation interval of $\Phi$ on $A$} is
defined as
\begin{equation}
  \label{e.rotation_interval}
  \rho_A(\Phi) \ = \ \left\{\rho\in\R\mid \exists z_i\in\pi^{-1}(A), n_i\nearrow \infty: 
    \iLim \pi_1\left(\Phi^{n_i}(z_i)-z_i\right)/n_i = \rho\right\} \ .
\end{equation}
When $\rho_A(\Phi)$ is reduced to a singleton $\{\rho\}$, we say
$\Phi$ has a unique rotation number $\rho$ in $A$. In this case,
$\lim_{n\to \infty} \pi_1(\Phi^n(z)-z)/n=\rho$ for all $z\in
\pi^{-1}(A)$. When this happens we say that $\varphi$ has a
well-defined rotation number $\rho(\varphi)=\rho+\Z\in \kreis$.

Given metric spaces $X,Y$, a continuous map $\varphi:X\to X$ is {\em
  semiconjugate} to $\psi:Y\to Y$ if there exists a continuous onto
map $h:X\to Y$ such that $h\circ \varphi=\psi\circ h$. In this
situation, we say $\psi$ is a {\em factor} of $\varphi$ and $h$ is a
{\em semiconjugacy} or {\em factor map}.  An important case is that of
{\em monotone} semiconjugacies.  A continuous map $h:X\to Y$ is called
{\em monotone} if all fibres $h^{-1}(\{y\}),\ y\in Y$, are connected.
A set $U\ssq X$ is called {\em saturated} with respect to $h:X\to Y$,
if $x\in U$ implies $h^{-1}(\{h(x)\})\ssq U$. If $h$ is continuous,
then it maps saturated open (closed) sets to open (closed) sets. As a
direct consequence, we have
\begin{lem} \label{l.monotone_preimage} Preimages of connected sets under
  surjective monotone maps are connected. In particular, preimages of
  decomposable sets are decomposable.
\end{lem}
A \emph{cellular continuum} in a surface $S$ is one of the form
$K=\bigcap_{n\in \N} D_n$ where each $D_n$ is a closed topological
disk and $D_{n+1}\subseteq \inter D_n$. This is equivalent to saying
that $K$ is a continuum and has a neighborhood homeomorphic to $\R^2$
in which $K$ is non-separating.

A partition $\cF$ of a metric space $X$ into compact subsets is called
an {\em upper semicontinuous decomposition} if for each open set
$U\subseteq X$, the union of all elements of $\cF$ contained in $U$ is
also open. A {\em Moore decomposition} of a surface $S$ is an upper
semicontinuous decomposition of $S$ into cellular continua.  The
following version of Moore's theorem is contained in \cite[Theorem
25.1]{daverman} (see also Theorem 13.4 in the same book). It says
essentially that the quotient space of a Moore decomposition is the
same surface $S$.

\begin{thm}\label{t.moore}
  Given any Moore decomposition $\cF$ of a surface $S$, there exists a map $\Pi:S\to S$
  which satisfies the following.  \romanlist
\item $\Pi$ is continuous and surjective;
\item $\Pi$ is homotopic to the identity (and preserves orientation if $S$ is orientable);
\item For all $z\in S$, we have $\Pi^{-1}(z)\in\cF$. 
 \listend
\end{thm}

The map $\Pi$ is called the {\em Moore projection} associated to $\cF$. 

Finally, we state some basic results from plane topology. We say that
a subset $K\subseteq X$ of a topological space separates two points if
the two points belong to different connected components of $X\sm
K$. 

\begin{lem}[{\cite[Theorem
    14.3]{Newman1992PlaneTopology}}]\label{l.component_separation} If
  two points in the plane are separated by a closed set, then they are
  also separated by some connected component of that set.
\end{lem}

\begin{lem}[{\cite[Theorem 2-28]{HockingYoung1961Topology}}]\label{l.curve-criterion} 
  In any metric space, a continuum $K$ is homeomorphic to a circle if
  for any pair $x\neq y$ of points of $K$, the set $K\sm \{x,y\}$ is
  disconnected.
\end{lem}

\begin{lem}[{\cite[Theorem 2-16]{HockingYoung1961Topology}}]\label{l.toplemma} 
  If $X$ is a continuum and $Y\subseteq X$ is closed, then the closure
  of every connected component of $X\sm Y$ intersects $Y$.
\end{lem}

\section{Minimal generators}

Throughout this section, $B\ssq\A$ denotes a decomposable cobasin
boundary and $\cB=\pi^{-1}(B)$ its lift. If $G\ssq\cB$ is a continuum
such that $\cB =\bigcup_{n\in\Z} T^n(G)$, we say that $G$ is a
\emph{generator} of $\cB$. We say that $G$ is a {\em minimal
  generator} if it does not strictly contain a smaller generator. In
the same way we may define generators and minimal generators for lifts
of circloids. This concept has been used implicitly by Barge and
Gillette in \cite{BargeGillette1991CofrontierRotationSets}; the
terminology is taken from \cite{JaegerPasseggi2013THIrrational}. As a
consequence of Zorn's Lemma, any generator contains a minimal
generator. The aim of this section is to provide a number of basic
facts on minimal generators which will be crucial for the later
constructions. The main objective is to derive the statements for
circloids, but in order to do so first have to consider cobasin
boundaries.

\begin{lem}\label{l.gen} A continuum $G\subseteq \cB$ is a generator of $\cB$ if
  and only if $G\cap TG\neq \emptyset$.
\end{lem}
\begin{proof}
  If $G\cap TG\neq \emptyset$, then $\bigcup_{k\in \Z} T^kG\subsetneq
  \cB$ is horizontally separating and by Lemma \ref{l.minimal_strips}
  it has to be equal to $\cB$, so $G$ is a generator.  To prove the
  converse, we first note that if $G$ is a generator then $G\cap
  T^kG\neq \emptyset$ for some $k>0$, since otherwise $G$ would
  project injectively onto $B\subseteq \A$ contradicting the fact that
  $B$ is essential. If $k=1$ we are done; otherwise assume that $k$ is
  maximal with the property that $G\cap T^kG\neq\emptyset$, and note
  that $\bigcup_{i\in \Z} T^{ik}G$ is horizontally separating and so
  must be equal to $\cB$; in particular it contains $TG$. But $TG\cap
  T^{ik}G=\emptyset$ for $i<0$ and $i>1$ due to the maximality of $k$.
  Thus $TG\subseteq G\cup T^kG$, and since $G$ is compact, $TG$ cannot
  be contained in $T^kG$, so $G\cap TG\neq \emptyset$ as claimed.
\end{proof}

\begin{lem}\label{l.strip_sep} Suppose that $L$ and $R$ are closed
  connected subsets of $\cB$, with $L$ unbounded to the left and $R$
  unbounded to the right. If $L\cap R= \emptyset$, then $\cB\sm (L\cup
  R)$ is connected, and if $L\cap R\neq\emptyset$ then $L\cup R=\cB$.
\end{lem}
\begin{proof}
  If $L\cap R \neq \emptyset$, then $L\cup R$ is horizontally
  separating, so by Lemma \ref{l.minimal_strips} it must be equal to
  $\cB$. Assume that $L\cap R=\emptyset$.  Let $W^-$ and $W^+$ be the
  connected components of $\R^2\sm \cB$ which are unbounded below and
  above, respectively, so $\cB= \bd W^-=\bd W^+$. Note that $L\cup R$
  cannot be horizontally separating (since this would contradict Lemma
  \ref{l.minimal_strips}); thus $W^-$ and $W^+$ are contained in the
  same connected component $U$ of $\R^2\sm (L\cup R)$. Since
  $E:=\cB\sm (L\cup R)\subseteq \bd W^-\cap \bd W^+$, it follows that
  $E\subseteq U$. Note that $U$ is simply connected, since both $L$
  and $R$ are connected and unbounded. Moreover, $E$ is a closed
  subset in the topology of $U$, and since $\cB$ separates $W^-$ from
  $W^+$ in $\R^2$ it follows that $E=\cB\cap U$ separates $W^-$ from
  $W^+$ in $U$. By Lemma \ref{l.component_separation} applied to
  $U\simeq \R^2$, some connected component $E_0$ of $E$ separates
  $W^-$ from $W^+$ in $U$. Since $E_0$ is closed in $U$, we have that
  $L\cup E_0\cup R$ is closed and horizontally separating, so by Lemma
  \ref{l.minimal_strips} it must be equal to $\cB$. This implies that
  $E_0=E$, so $E$ is connected.
\end{proof}

\begin{lem}\label{l.xy} There exists a minimal generator $G_0$ such
  that $G_0\cap T^kG_0\neq \emptyset$ if and only if $|k|\leq 1$.
  Moreover, $G_0\sm TG_0$ and $G_0\sm(TG_0\cup T^{-1}G_0)$ are connected and
  dense in $G_0$.
\end{lem}
\begin{proof} 
  Since $B$ is decomposable, there exists a decomposition $B=X\cup Y$
  into proper subcontinua.  As $B$ is a cobasin boundary, both $X$ and
  $Y$ must be inessential in $\A$, which implies that there are open
  topological disks $D_X\supset X$ and $D_Y\supset Y$. Let $\til{D}_X$
  be a connected component of $\pi^{-1}(D_X)$, and $\cX =
  \pi^{-1}(X)\cap \til{D}_X$. Let $\til{D}_Y$ be a connected component
  of $\pi^{-1}(D_Y)$, and $\cY = \pi^{-1}(X)\cap \til{D}_Y$. The sets
  $\cX$ and $\cY$ project injectively onto $X$ and $Y$, respectively,
  so they are continua and $T^i\cX\cap \cX=\emptyset$ for all $i\neq
  0$, and similarly for $\cY$. Since $X\cap Y\neq \emptyset$, there
  exists $n\in \Z$ such that $T^n\cY\cap \cX\neq \emptyset$. The set
  $\cX \cup T^n\cY$ is a generator, so it contains some minimal
  generator $G_0$. Let $k$ be the largest integer such that $G_0\cap
  T^kG_0\neq \emptyset$, and suppose for a contradiction that $k>1$.
  Then $\bigcup_{i\in \Z,\, i\neq 1} T^iG_0\subseteq \bigcup_{i\in
    \Z,\, i\neq 1} T^i\cX \cup T^{n+i}\cY$ is closed and horizontally
  separating, so in particular it contains $T\cX$. Since $T\cX$ is
  disjoint from $T^i\cX$ for all $i\neq 1$, it follows that $T\cX
  \subseteq \bigcup_{j\in \Z} T^{j}\cY$, which implies that $X=\pi(\cX) \subset Y$, contradicting our choice of $X$ and $Y$.

  By Lemma \ref{l.strip_sep} we have that $E_0 := \cB \sm
  \bigcup_{k\neq 1} T^kG_0= G_0\sm (T^{-1}G_0\cup TG_0)\neq \emptyset$ is
  connected. Note that the closure of any connected component of
  $G_0\sm TG_0$ intersects $TG_0$ (see Lemma \ref{l.toplemma}), so any
  connected component of $G_0\sm TG_0$ must contain $E_0$ (otherwise
  it would be contained in $T^{-1}G_0$ which is disjoint from $TG_0$),
  so there is only one such component. Thus $G_0\sm TG_0$ is
  connected. Since $\bigcup_{k\in \Z} T^k(G_0\sm TG_0) = \cB$, it
  follows that $\ol{G_0\sm TG_0}$ is a generator and by minimality
  $G_0 = \ol{G_0\sm TG_0}$. This implies that $C=G_0\cap TG_0$ has
  empty interior in the restricted topology to $G_0$, and therefore
  $E_0 = G_0\sm (T^{-1}G_0\cup TG_0) = G_0\sm (C\cup TC)$ is also
  dense in $G_0$, completing the proof.
\end{proof}

\begin{lem}\label{l.gen_inter} If $G$ is any minimal generator of
  $\cB$, then $T^kG\cap G\neq \emptyset$ if and only if $|k|\leq 1$.
  Moreover, $G\sm TG$ and $G\sm (TG\cup T^{-1}G)$ are connected and
  dense in $G$.
\end{lem}
\begin{proof} 
  Suppose that $T^kG\cap G\neq \emptyset$ for some $k>1$, so there
  exists $z\in G$ such that $T^kz\in G$. Let $G_0$ be as in Lemma
  \ref{l.xy}. Replacing $G$ by $T^iG$ for a suitable $i$, we may
  assume that $z\in T^{-1}G_0$. This means that $G$ intersects
  $T^{-1}G_0$ and $T^{k-1}G_0$, where $k-1\geq 1$. Thus the set $G\cup
  \bigcup_{k\neq 0} T^iG_0$ is closed, connected and horizontally
  separating, and by Lemma \ref{l.minimal_strips} it should be equal
  to $\cB$. Thus $G_0\sm (T^{-1}G_0\cup TG_0) = G_0\sm \bigcup_{k\neq
    0} T^iG_0 \subseteq G$, implying that $G_0\subseteq G$. By
  minimality $G_0=G$, contradicting the fact that $G\cap T^kG\neq
  \emptyset$ with $k>1$.

  Knowing that $G\cap T^kG\neq \emptyset \iff |k|\leq 1$, the
  remaining claims are proved exactly as in the last paragraph of the
  proof of Lemma \ref{l.xy}.
\end{proof}

Given a minimal generator $G$ of $\cB$, let $\cL_n(G)=\bigcup_{k\leq
  n} T^{k}(G)$ and $\cR_n(G)=\bigcup_{k\geq n} T^{k}(G)$. With these
notions, we have

\begin{lem} \label{l.generator_inequality} If $G$ and $G'$ are two
  different minimal generators of $\cB$, then either
  $G'\ssq\cB\smin\cL_{-1}(G)\ssq\cR_0(G)$ or $G'\ssq
  \cB\smin\cR_1(G)\ssq\cL_0(G)$.
\end{lem}
\begin{proof}
  Suppose $G'$ intersects both $\cL_{-1}(G)$ and $\cR_1(G)$. Then
  $\cL_{-1}(G)\cup G'\cup \cR_1(G)=\cB$ by Lemma
  \ref{l.minimal_strips}, so the union contains $G$. Since
  $G\sm(TG\cup T^{-1}G)$ is disjoint from $\cL_{-1}(G)\cup \cR_1(G)$,
  we have that $G\sm (TG\cup T^{-1}G)\subseteq G'$, which implies by
  Lemma \ref{l.gen_inter} that $G\subseteq G'$, so by the minimality
  $G=G'$.
\end{proof}

\begin{cor} \label{c.generator_ends} If $G$ and $G'$ are minimal
  generators of $\cB$, then $G'$ is contained in two adjacent copies
  of $G$ and vice versa.
\end{cor}

A \emph{cut} (of $\cB$) is a set of the form $G\cap T^{-1}G$ where $G$
is a minimal generator of $\cB$. We denote by $\cC$ the family of all
cuts. Note that by Lemma \ref{l.generator_inequality}, cuts are
pairwise
disjoint. 

Given a cut $C=G\cap T^{-1}(G)$, we let $\wh\cR(C)=\cR_0(G)$ and
$\wh\cL(C)=\cL_{-1}(G)$, so that $C=\wh\cR(C)\cap \wh\cL(C)$. Further,
we let $\cL(C)=\wh\cL(C)\smin C$ and $\cR(C)=\wh\cR(C)\smin C$. We
write $C\prec C'$ if $C\ssq \cL(C')$, or equivalently if $C'\ssq
\cR(C)$. By Lemma \ref{l.generator_inequality} and its corollary,
$\prec$ defines a total order in $\cC$. We extend this notation to
compare arbitrary subsets $S\ssq\cB$ with cuts by writing $S\prec C$
if $S\ssq \cL(C)$ and $S\preccurlyeq C$ if $S\ssq\wh\cL(C)$. If $z\in
\cB$ and $\{z\}\prec C$, we simply write $z\prec C$. For two cuts
$C\prec C'$, we let $(C,C')_\cB = \cR(C)\cap\cL(C')=\{z\in\cB\mid
C\prec z\prec C'\}$ and $[C,C']_\cB=\wh\cR(C)\cap
\wh\cL(C')=(C,C')_\cB\cup C\cup C'$.

We note that cuts need not be connected. However, we have:
\begin{lem}
  \label{l.interval_connectedness}
  Given two cuts $C\prec C'$, the set $(C,C')_{\cB}$ is connected and
  its closure is $[C,C']_\cB$.
\end{lem}
\begin{proof}
  We have $(C,C')_{\cB} = \cB\sm (\cL_0(G) \cup \cR_0(G'))\neq
  \emptyset$ for some minimal generators $G,G'$. The connectedness
  follows from Lemma \ref{l.strip_sep}, and the fact that its closure
  is $[C,C']_{\cB}$ follows easily from Lemma \ref{l.gen_inter}.
\end{proof}
Let $A$ be the essential circloid such that $\bd A=B$ (\ie $A$ is the union
of $B$ with all bounded connected components of $\A\sm B$), and let
$\cA=\pi^{-1}(A)$ be its lift. A generator of $\cA$ is a continuum
$G\ssq \cA$ which satisfies $\bigcup_{n\in \Z}T^nG=\cA$.  In order to go
over from a the decomposable cobasin boundary $B$ to the corresponding
circloid $A$, the following statements will be crucial.

\begin{lem}\label{l.component_bounded} All connected components of
  $\cA\sm \cB$ are topological disks with diameter bounded by a
  uniform constant $M$.
\end{lem}
\begin{proof}
  Let $G$ be a generator of $\cB$, and suppose $\pi_1(G)=[a,b]$. Let
  $N\in \N$ be such that $a+N>b$. Since $\pi_2(G) = \pi_2(\cA)$, the
  latter set has diameter bounded by some constant $c$. let
  $M=2N+c+1$. If $U$ is a connected component of $\cA\sm \cB$ with
  $\diam(U)>M$, then $\diam(\pi_1(U)) > 2N+1$ and we may assume
  $[a-N,b+N]\subseteq \pi_1(U)$ replacing $U$ by $T^iU$ for an
  appropriate $i\in \Z$. Thus there is a simple arc
  $\gamma\colon[0,1]\to U$ such that $\pi_1(\gamma(0))=a-N$,
  $\pi_1(\gamma(1))=b+N$, and $a-N<\gamma(t)<b+N$ for $0<t<1$. If $K=
  \gamma([0,1])$ and $S:=[a-N, b+N]\times \R$, we have that $S\sm K$
  has exactly two connected components $S_+$ and $S_-$, the former
  unbounded above and the later unbounded below. The set
  $G'=\bigcup_{k=-N}^N T^kG$ being connected, disjoint from $K$ and
  contained in $S$, must lie entirely in $S^-$ or $S^+$. Suppose
  without loss of generality that $G'\subseteq S^-$. Note that
  $K\subseteq U$ and $U$ is is bounded above. If $y$ is the smallest
  real such that $\{a\}\times [y,\infty) \subseteq \R^2\sm \cB$, then
  $\{a\}\times [y,\infty)$ is disjoint from $K$ and thus contained in
  $S^+$, and since $z:=(a,y)\in \bd \cB$ there must exist $n$ such
  that $z\in T^nG$. Since $T^kG\subseteq S^-$ when $|k|\leq N$, it
  follows that $n>N$, but this is not possible since $\pi_1(T^nG)
  \subseteq [a+N, \infty) \subseteq (b,\infty)$.
\end{proof}

Given a continuum $S\ssq \R^2$, the complement $\R^2\smin S$ consists
of one unbounded component and a union of topological disks. We denote
the unbounded component by $U_\infty(S)$ and the family of disks by
$\cU_S$ and let $S^{\filled}=\R^2\smin U_\infty(S)=
S\cup\bigcup_{U\in\cU_S}U$. Note that $S^\filled$ is a nonseparating
continuum.

\begin{lem} \label{l.disk_boundary} Suppose $U$ is a connected
  component of $\cA\smin \cB$ and $C^-_0\prec C^+_0$ are cuts such
  that $(C^-_0,C^+_0)_\cB\cap\partial U\neq \emptyset$. If $C^-,C^+$
  are cuts with $C^-\prec C^-_0\prec C^+_0\prec C^+$, then $\bd{U}\ssq
  [C^-,C^+]_\cB$.
\end{lem}
\begin{proof}
  Suppose for a contradiction that $\bd{U}\not\subseteq[C^-,C^+]_\cB$.
  Assume without loss of generality that $\bd{ U}\cap \cL(C^-) \neq
  \emptyset$. Since $U$ is bounded, there is $n>0$ such that $\bd
  U\subseteq [T^{-n}C^-, T^n C^-]_\cB$, and we may assume $T^nC^-\prec
  C_0^-$ and $C_0^+\prec T^nC^-$. The sets $L = [T^{-n}C^-,
  C_0^-]_{\cB}$ and $R=[C^-, T^nC^-]_{\cB}$ are connected, and $L\cap
  R = [C_0^-, C^-]_{\cB}$ is also connected (by Lemma
  \ref{l.interval_connectedness}). Since $\bd U$ is not contained in
  either set $L$ or $R$, we have that $U\subseteq U_\infty(L)\cap
  U_\infty(R)$. But then by Lemma \ref{l.component_separation} we have
  that $U\subseteq U_\infty(R\cup L)$, contradicting the fact that $U$
  is bounded and $\bd U\subseteq L\cup R$.
\end{proof}

\section{Dynamical linearisation: Proof of Theorem~\ref{t.poincare}}

Throughout this section, we assume that $\varphi\colon \A\to \A$ is a homeomorphism homotopic to the identity and $A$ is a $\varphi$-invariant circloid with decomposable boundary. We let $B=\partial B$ and denote the
lifts of $A,\varphi$ and $B$ by $\cA,\Phi$ and $\cB$, respectively. The next result generalizes \cite[Theorem 2.7]{BargeGillette1991CofrontierRotationSets} to circloids.

\begin{lem} The rotation number $\rho(\Phi,\cA) = \lim_{n\to \infty}
  \pi_1(\Phi^n(z)-z)/n$ exists and is independent of $z\in \cA$.
\end{lem}
\begin{proof}
  The fact that $\rho = \lim_{n\to \infty} \pi_1(\Phi^n(z_0)-z_0)/n$
  exists for some $z_0\in \cB$ follows from the Birkhoff Ergodic
  Theorem and the existence of an invariant measure for $\varphi|_B$,
  since $\pi_1(\Phi^n(z_0)-z_0)/n$ is a Birkhoff average for the
  function $B \ni x\mapsto \pi_1(\Phi(x') - x')\in \R$, where $x'\in
  \pi^{-1}(x)$ is arbitrary.

  Fix a minimal generator $G$ of $\cB$ containing $z_0$. Since
  $\Phi^n(G)$ is also a minimal generator, by Corollary
  \ref{c.generator_ends} there exists $k$ such that $\Phi^n(G)\subset
  T^k(G\cup TG)$, so $\diam(\Phi^n(G)) \leq 2\diam(G):=M$ for all
  $n\in\N$. This implies
  $\lim_{n\to\infty}\pi_1\left(\Phi^n(z)-z\right)/n=\rho$ for all $z\in G$,
    and since $G$ is a generator the same holds for all $z\in\cB$.
    Finally, we deduce from Lemma \ref{l.component_bounded} that the
    same property holds for $z\in \cA$.
\end{proof}

\begin{lem} $\rho(\Phi,\cA)=p/q\in \Q$ if and only if there exists
  $z\in \cA$ such that $\Phi^q(z)=T^p z$.
\end{lem}
\begin{proof}
  The if-part is trivial. For the other implication, note that it is
  easy to verify that $\rho(T^{-p}\Phi^q,\cA) = 0$ if and only if
  $\rho(\Phi,\cA) = p/q$, so it suffices to assume that
  $\rho(\Phi,\cA)=0$ and show that there is a fixed point in $\cA$.
  Fix a minimal generator $G$ of $\cB$. We claim that $\Phi^n(G)\cap
  G\neq \emptyset$ for all $n$. Indeed, fix a positive integer $n$. Corollary
  \ref{c.generator_ends} implies that $\Phi^n(G)\subset T^kG\cup
  T^{k+1}G$ for some $k$. If $k>0$, then $\Phi^n(G)\subset \cR_1(G)$
  which then implies $\Phi^{nk}(G)\subset \cR_k(G)$ and this implies
  that $\pi_1(\Phi^{nk}(z)-z)/{nk} \geq 1/n$ for all $k\in \N$ and $z\in G$,
  contradicting our assumption. If $k< -1$, we get a similar
  contradiction. Thus $k\in \{0,-1\}$, and since $G$ and $\Phi^n(G)$
  are minimal generators, $\Phi^n(G)$ must intersect both $T^kG$ and
  $T^{k+1}G$, so $\Phi^n(G)\cap G\neq \emptyset$ as claimed. Thus $K = \ol{\bigcup_{n\in \Z}
    \Phi^n(G)}\subset \cB$ is connected and bounded (again due to
  Corollary \ref{c.generator_ends}), and $\Phi(K)=K$. Moreover,
  $K^\filled\subset \cA$ and $\Phi(K^\filled)=K^\filled$, so
  $K^\filled$ is a non-separating invariant continuum in $\R^2$ and the generalization of the 
  Cartwright-Littlewood theorem due to Bell \cite{bell} implies that $K^\filled$ contains a fixed point of $\Phi$.
\end{proof}

The previous lemma implies the first item from Theorem~\ref{t.poincare}. It also follows that if $\varphi|_A$ is monotonically semiconjugate to an irrational rotation then $\rho(\Phi, \cA)$ is irrational. Thus to complete the proof of the theorem it remains to prove that if $\rho = \rho(\Phi, \cA)$ is irrational, then $\varphi|_A$ is monotonically conjugate to the corresponding irrational rotation.  

For the remainder of this section we assume that $\rho$ is irrational, and we fix a minimal generator $G_0$ of $\cB$.  Given $x=n+k\rho$ in $Q(\rho)=\{l+\rho m\mid l,m\in\Z\}$,
we let $G_x = \Phi^k\circ T^n(G_0)$ and denote by $C_x=G_x\cap
T^{-1}(G_x)$ the cut corresponding to $G_x$. Recall that there is a linear ordering $\prec$ on the set $\cC$ of cuts as defined in the previous section.

\begin{lem}\label{l.generator_monotonicity}
  The mapping $(Q(\rho),<)\to(\cC, \prec),\ x\mapsto C_x$ is strictly monotonically
  increasing. In particular, $C_x\cap C_y=\emptyset$ if $x\neq y$.
  Moreover, the set $[C_x,C_y]_\cB$ is connected for all $x<y$,
  decreasing in $x$ and increasing in $y$.
\end{lem}
\proof Suppose that $x=n+k\rho<x'=n'+k'\rho$, but $C_{x'}\preccurlyeq
C_x$. Then $\Phi^{k'-k}\circ T^{n'-n}(\wh\cL(C_0))\ssq \wh\cL(C_0)$.
As a consequence, all orbits in $\wh\cL(C_0)$ under the lift
$\Psi=\Phi^{k'-k}\circ T^{n'-n}$ of $\varphi^{k'-k}$ are bounded to
the right, contradicting the fact that $\rho(\Psi)=x'-x>0$. This shows
the strict monotonicity of $x\mapsto C_x$ and the disjointness.
Connectedness of $[C_x,C_y]_\cB$ is given by
Lemma~\ref{l.interval_connectedness}.  \qed\medskip

As in the previous section, given $z\in\cB$ and a cut $C\in\cC$, we
write $z\prec C$ iff $z\in\cL(C)$ and $C\prec z$ iff $z\in\cR(C)$. In
order to extend this notion to all $z\in \cA$, note that $\cA\smin\cB$
is a union of bounded open topological disks whose boundary is
contained in $\cB$. Given $z\in\cA\smin\cB$, we denote the respective
disk containing $z$ by $U_z$ and write $z\prec C$ iff $\partial
U_z\ssq \cL(C)$ and $C\prec z$ iff $\partial U_z\ssq\cR(C)$.
Equivalently, $z\prec C$ iff $z\in\cL(C)^\filled$ and $C\prec z$ iff
$z\in\cR(C)^\filled$. Given a subset $S\ssq\cA$, we write $C\prec S$
iff $S\ssq\cL(C)^\filled$ and $S\prec C$ iff $S\ssq\cR(C)^\filled$.
Then, we define $H:\cA\to \R$ by
\begin{equation}
  H(z)\  = \  \sup\{x \in Q(\rho)\mid C_x
  \prec z\} \ .
\end{equation}
\begin{lem}
  The map $H$ is continuous and projects to a monotone semiconjugacy
  $h:A\to\kreis$ from $\varphi|_A$ to the irrational rotation by $\rho$.
\end{lem}

\begin{proof}
We first show
the continuity of the restriction of $H$ to $\cB$. By definition, we have that
\begin{equation} H_{|\cB}^{-1}((x,y))\ =\ \bigcup_{\stackrel{x',y'\in
    Q(\rho)}{x<x'<y'<y}} (C_{x'},C_{y'})_\cB \ = \ \bigcup_{\stackrel{x',y'\in
    Q(\rho)}{x<x'<y'<y}} [C_{x'},C_{y'}]_\cB \ .
\end{equation}
Since the sets $(C_{x'},C_{y'})_\cB$ are relatively open in $\cB$, this shows that
preimages of open sets are open, so that $H_{|\cB}$ is continuous. 

In order to see that $H$ is continuous on all of $\cA$, fix $x<y$ and $z\in
H^{-1}((x,y))$. It suffices show to that $H^{-1}((x,y))$ contains a neighbourhood of
$z$. If $z\in \cA\smin\cB$, then by definition the whole open disk $U_z$ is
contained in $H^{-1}((x,y))$. (Note that if $C_x\prec z$, then $C_x\prec z'$ for
all $z'\in U_z$.) Thus, suppose that $z\in\cB$. Since $H_{|\cB}$ is continuous,
there exists $\eps>0$ such that $B_\eps(z)\cap \cB\ssq H^{-1}((x,y))$. If $z'\in
B_\eps(z)\cap(\cA\smin\cB)$, then $B_\eps(z)$ intersects both $\cB$ and
$U_{z'}$. Consequently, $B_\eps(z)$ intersects $\partial U_{z'}$ and we have
$\emptyset \neq \partial U_{z'}\cap B_\eps(z)\ssq H^{-1}((x,y))$. However, this
implies that for some $x',y'$ with $x<x'<y'<y$ we have $\partial U_{z'}\cap
(C_{x'},C_{y'})_\cB\neq \emptyset$. Therefore Lemma~\ref{l.disk_boundary} yields that
$\partial U_{z'}\ssq [C_{\tilde x},C_{\tilde y}]_\cB$ for any $\tilde x,\tilde y$
with $x<\tilde x<x'<y'<\tilde y<x$ and thus $U_{z'}\ssq
H^{-1}((x,y))$. Altogether, we obtain $B_\eps(z)\cap\cA\ssq H^{-1}((x,y))$, which
proves the continuity of $H$ on $\cA$.
 
In order to show the further statements, note that since by definition
$\Phi(G_x)=G_{x+\rho}$ and $T(G_x)=G_{x+1}$, the same relations hold for $C_x$
and $\cL(C_x)$. Using these facts, it is easy to check that $H$ is a
semiconjugacy from $\Phi_{|\cA}$ to the translation $x\mapsto x+\rho$ on $\R$
and that $H$ commutes with the deck translation $T$, such that $H$ projects to a
semiconjugacy $h$ from $\varphi_{|A}$ to the rotation $R_\rho$.

It remains to prove the monotonicity of $h$, which will follow immediately from
that of $H$.  We have that
\begin{equation}
  H_{|B}^{-1}(x) \ =\  \bigcap_{\stackrel{x',y'\in Q(\rho)}{x'<x<y'}} [C_{x'},C_{y'}] \ .
\end{equation}
This can be seen as a nested intersection of continua and is therefore
a continuum itself. The full fibre $H^{-1}(x)$ is obtained by adding
the union $\bigcup_{z\in H^{-1}(x)\smin\cB} U_z$ of topological open
disks to $H_{|B}^{-1}(x)$. However, if $z\in H^{-1}(x)$, then
$\partial U_z$ cannot intersect $\cL(C_{x'})$ for any $x'<x$, since
otherwise Lemma~\ref{l.disk_boundary} would imply that $\partial
U_z\ssq \wh\cL(C_{\tilde x})$ for some $\tilde x\in (x',x)$ and thus
$H(z)\leq \tilde x$. Similarly, $\partial U_z$ is disjoint from
$\cR(C_{y'})$ for all $y'>x$, and therefore $\partial U_z\ssq
H^{-1}(x)$. Hence, we have that $H^{-1}(x)=H_{\cB}^{-1}(x)^\filled$ is
the `fill-in' of a continuum, and hence a continuum itself.
\end{proof}

This concludes the proof of Theorem \ref{t.poincare}.

\section{Topological linearisation: a universal factor map}
\label{UniversalFactorMap}

The main goal of this section is to prove Theorem \ref{t.universal_factor_map}.
We refer the reader to Section~\ref{Kuratowski} for a
discussion on the relationship between the content presented here and the work of Kuratowski.

As before, we suppose $A\ssq\A$ is a circloid with
decomposable boundary, and $\cA=\pi^{-1}(A)$.  We let $B=\partial A$
and $\cB=\pi^{-1}(B)$.  Our aim is to define a Moore decomposition of
$\A$ such that the corresponding projection maps $A$ to a topological
circle. As before, we start by decomposing $B$, and we use the
family $\cC$ of cuts of $\cB$ as the main tool. 

However, cuts need not be connected, and moreover is easy to give
examples where $\bigcup_{C\in\cC} C$ does not cover all of $\cB$.  In
order to obtain a decomposition starting from $\cC$, we define a
strong partial order relation $\llcurly$ on $\cC$ by writing
$C\llcurly C'$ if and only if there exist uncountably many cuts
$\tilde C\in \cC$ such that $C\prec \tilde C\prec C'$.  Similar to
before, we extend this definition to arbitrary subsets $S,S'\ssq\cA$
by writing $S\llcurly S'$ whenever there exist uncountably many cuts
$\tilde C$ with $S\prec \tilde C\prec S'$.\foot{We note that the
  requirement of uncountably many intermediate cuts is crucial for the
  whole construction in this section. We do not elaborate further on
  this, but just mention that a examples demonstrating why requiring
  uncountably many intermediate cuts is necessary can be produced by
  gluing finitely or countably many pseudoarcs together. Otherwise, a
  statement analogous to Lemma~\ref{l.intermediate_cuts} does not
  hold.}  In case of one-point sets $S=\{z\}$, we write $z\llcurly S'$
instead of $\{z\}\llcurly S'$. Then, given any $z\in\cB$, we define
\begin{equation}
  \label{e.fibres}
  F(z) \ = \ \bigcap_{\stackrel{C^-,C^+\in\cC}{C^-\llcurly z\llcurly C^+}} [C^-,C^+]_\cB \ .
\end{equation}
We let $\cF_\cB=\{F(z)\mid z\in \cB\}$ and call the elements
$F\in\cF_{\cB}$ {\em fibres} of $\cB$. Further, we let
$\cF_B=\{\pi(F)\mid F\in\cF_{\cB}\}$. We note that the intersection in
(\ref{e.fibres}) can be viewed as a nested intersection of continua:
by compactness, for every $n\in\N$ there exist $C_n^-\llcurly z
\llcurly C^+_n$ such that $F(z)\ssq [C^-_n,C^+_n] \ssq
B_{1/n}(F(z))$. Without loss of generality we may assume that
$C^-_n\llcurly C^-_{n+1}$ and $C^+_{n+1}\llcurly C^+_n$ for all
$n\in\N$, and we have $F(z)=\bigcap_{n\in\N} [C^-_n,C^+_n]$. By Lemma \ref{l.interval_connectedness} this is a
nested intersection of continua, hence a continuum.

\begin{lem}\label{l.intermediate_cuts}
  If $S,S'\ssq\cB$ and $S\llcurly S'$, then there exists $C\in\cC$ such that
  $S\llcurly C\llcurly S'$.
\end{lem}
\proof We first show the following slightly weaker
\begin{claim} \label{c.intermediate_cuts}
  If $V,V'\ssq\cA$ and $V\llcurly V'$, then there exist $C,C'\in\cC$ such that
  $V\llcurly C\prec V'$ and $V\prec C'\llcurly V'$.
\end{claim}\changenotsign
We will prove the existence of $C$; that of $C'$ then follows by
symmetry. Suppose for a contradiction that for all $C\in\cA$ with $V\prec C\prec
V'$ we have $V\not\llcurly C$. Let $R=\bigcap_{C\in\cC,C\prec V'} \wh\cR(C)$ and
choose an increasing sequence of cuts $C_n\prec V'$ with $R=\bigcap_{n\in\N}
\wh\cR(C_n)$. Note that, for example, it suffices to choose $C_n$ such that
$\dist_{\cH}(\wh\cR(C_n),R)<1/n$. Then, since $V\not\llcurly C_n$, there exist
at most countably many cuts between $V$ and $C_n$. However, every cut between
$V$ and $V'$ is either equal to $C_n$ for some $n\in\N$, lies between $V$ and
$C_1$ or lies between $C_n$ and $C_{n+1}$ for some $n\in\N$. Altogether, we
obtain that there are at most countably many cuts between $V$ and $V'$,
contradicting $V\llcurly V'$. This proves the claim.\smallskip

Now, suppose for a contradiction that for every cut $C$ between $S$
and $S'$ we either have $S\not\llcurly C$ or $C\not\llcurly S'$. Note
that both properties cannot hold simultaneously since $S\llcurly S'$. Let
$L'=\bigcap_{C\in\cC,S\llcurly C} \wh\cL(C)$ and
$R'=\bigcap_{C'\in\cC,C'\llcurly S'}\wh{\cR}(C')$. Then the
intersection $E=L'\cap R'$ is non-empty, since for every pair of cuts
$C,C'$ with $ C'\llcurly S'$ and $S\llcurly C$ we have $C'\prec C$ due
to our contradiction assumption. Moreover, $G=T(L')\cap R'$ is a
continuum, since it can again be represented as a nested intersection
of intervals in $\cB$. Since $G\cap T^{-1}(G)=E\neq \emptyset$, it follows
from Lemma \ref{l.gen} that $G$ is a generator. Thus, it contains a minimal generator, and
consequently the set $E$ contains some cut $\til{C}$. If $S\llcurly
\til{C}$, then Claim \ref{c.intermediate_cuts} implies that there
exists $C_1\in \cC$ such that $S\llcurly C_1\prec \til{C}$, and since
$\til{C}\subseteq L'$ it follows that $\til{C}\subseteq
\wh{\cL}(C_1)$, contradicting the fact that $C_1\prec \til{C}$. Thus
$S\not\llcurly \til{C}$, and by a similar argument
$\til{C}\not\llcurly S'$.  However, this contradicts the fact that
$S\llcurly S'$. \qed\medskip

\begin{lem}\label{l.fibre_inequality}
  Distinct fibres $F,F'\in\cF_{\cB}$ are disjoint and either $ F\llcurly F'$ or
  $F'\llcurly F$.
\end{lem}
\proof Let $F=F(z)$ and $F'=F(z')$. If $F\neq F'$, then there exists
$C\in\cC$ such that either $z\llcurly C$ and $z'\not\llcurly C$, or
$C\llcurly z$ and $C\not\llcurly z'$. Assume the former case (the
other case is analogous). According to
Lemma~\ref{l.intermediate_cuts}, there exist $C^-,C^+$ with $z\llcurly
C^-\llcurly C^+\llcurly C$. However, this implies that $z\llcurly
C^-\llcurly C^+\llcurly z'$, hence $F(z)\ssq \wh\cL(C^-)$ and
$F(z')\ssq \wh\cR(C^+)$ are disjoint and $F(z)\llcurly F(z')$.
\qed\medskip

We now turn to the decomposition of $\cA$. As before, given $z\in\cA\smin\cB$,
we denote by $U_z$ the connected component of $\cA\smin\cB$ containing $z$.
\begin{lem}\label{l.fibres_and_disks}
  Suppose $F\in F_{\cB}$ intersects $\partial U_z$ for some
  $z\in\cA\smin\cB$. Then $\partial U_z\ssq F$.
\end{lem}
\proof Suppose $F=F(z')$ intersects $\partial U_z$. Then given any cuts
 $C^-\llcurly z\llcurly C^+$, Lemma~\ref{l.intermediate_cuts} yields cuts
$C_0^-,C_0^+$ with $C^-\llcurly C_0^-\llcurly F\llcurly C_0^+\llcurly C^+$ and
$\partial U_z\cap (C_0^-,C_0^+)_\cB\neq \emptyset$. By Lemma~\ref{l.disk_boundary},
this implies that $\partial U_z\ssq [C^-,C^+]_\cB$. Since this is true for all pairs
of cuts satisfying $C^-\llcurly z\llcurly C^+$, we obtain $\partial U_z\ssq
F$. \qed\medskip

Bounded connected components of the complement of a fibre
$F\in\cF_\cB$ are also bounded connected components of $\cA\smin\cB$.
Therefore, the preceding lemma implies that $F^\filled=F\cup
\bigcup_{\partial U_z\cap F\neq\emptyset} U_z$. This allows to define
a decomposition of $\cA$ by $\cF_\cA=\{F^\filled \mid F\in\cF_\cB\}$.
We denote fibres of $\cF_\cA$ by $\wh F$, and given $z\in\cA$ we let
$\wh F(z)$ be the unique fibre in $\cF_\cA$ which contains $z$. Note
that $F(z) = \bd \wh{F}(z)$, and so $\wh{F}(z)=(\bd
\wh{F}(z))^\filled$.

\begin{rem}
  An alternative way to define the fibres of $\cA$ is the following. Recall that
  we write $C\prec z$ for a cut $C$ and $z\in\cA$ iff $z\in\cR(C)^\filled$, and
  $z\prec C$ iff $z\in\cL(C)^\filled$. The notions $C\llcurly z$ and
  $z\llcurly C$ in  $\cF_\cA$ can then be defined as before by the existence of
  uncountably many intermediate cuts. Using this, the fibres of $\cA$ can be
  defined exactly in the same way as those of $\cB$ in
  (\ref{e.fibres}). Equivalence of the two definitions is provided by the
  following statement.
\end{rem}

\begin{lem}\label{l.filled_fibres} For any $z\in\cA$, we have 
\begin{equation}\label{e.filled_fibres}
\wh F(z) \ = \  \bigcap_{\stackrel{C^-,C^+\in\cC}{C^-\llcurly z\llcurly C^+}}
    [C^-,C^+]_\cB^\filled  \ .
  \end{equation}
\end{lem}
\proof If $z\in\cA\smin\cB$ and $z'\in\partial U_z$, then by
definition $\wh F(z)=\wh F(z')$, and the intersection on the right
hand side of (\ref{e.filled_fibres}) coincides as well. Thus, we may
assume $z\in\cB$. However, in this case the right side is just
$F(z)^\filled =\wh F(z)$ by Lemma~\ref{l.fibres_and_disks}.
\qed\medskip

Given $F^-,F^+\in \cF_\cB$, we let $(F^-,F^+)_\cB=\bigcup_{F^-\prec F
  \prec F^+} F = \{z\in\cB \mid F^-\prec z \prec F^+\}$ and
$[F^-,F^+]_\cB= F^-\cup (F^-,F^+)_\cB\cup F^+$.  Note that these
intervals are relatively open, respectively relatively closed, in $\cB$.
If $\wh F$ and $\wh F'$ are distinct fibres in $\cF_\cA$, then
$F=\partial\wh F$ and $F'=\partial F'$ are fibres in $\cF_\cB$.
According to Lemma~\ref{l.fibre_inequality}, we always have either
$F\prec F'$ or $F'\prec F$, so either $\wh F\prec \wh F'$ or $\wh
F'\prec \wh F$, and the notions $\prec$ and $\llcurly$ coincide for
the pairs $\wh F,\wh F'$ and $F,F'$.  As before, given $\wh F^-\prec
\wh F^+$, we let $(\wh F^-,\wh F^+)_\cA=\bigcup_{\wh F^-\prec \wh F
  \prec \wh F^+} \wh F$ and $[\wh F^-,\wh F^+]_\cA=\wh F^-\cup (\wh
F^-,\wh F^+)_\cA\cup \wh F^+$. As a consequence of Lemma~\ref{l.fibres_and_disks}, we
obtain that $(\wh F^-,\wh F^+)_\cA=(\partial \wh F^-,\partial\wh
F^+)_\cB^\filled$ and $[\wh F^-,\wh F^+]_\cA=[\partial \wh
F^-,\partial\wh F^+]_\cB^\filled$. In particular, this implies the
following observation, which we state for further use.

\begin{lem}
  Given $F^-,F^+$, we have that $(\wh F^-,\wh F^+)_\cA$ is
  relatively open and $[\wh F^-,\wh F^+]_\cA$ is relatively closed in
  $\cA$.
\end{lem}
Moreover, we have 
\begin{lem} \label{l.closed_fibre_intervals} For all $\wh F^-\prec \wh
  F^+\in\cF_\cA$, the set $[\wh F^-,\wh F^+]_\cA$ is a non-separating
  continuum.
\end{lem}
\proof 
If $\wh F^-=\wh F(z^-)$ and $\wh F^+=\wh F(z^+)$ with $z^-,z^+\in\cB$,
then 
\begin{equation}
\partial [\wh F^-,\wh F^+]_\cA \ = \ \bigcap_{\stackrel{C^-\llcurly z^-}{z^+\llcurly C^+}}
[C^-,C^+]_\cB \ .
\end{equation}
Hence, as a nested intersection of continua the set $\partial [\wh
F^-,\wh F^+]_\cB$ is connected, and so $[\wh F^-,\wh
F^+]_\cA=\left(\partial [\wh F^-,\wh F^+]_\cB\right)^\filled$ is a
filled -- and hence non-separating -- continuum.  \qed\medskip

\begin{lem}
  $\cF_{\cA}$ is an upper semicontinuous decomposition of $\cA$.
\end{lem}
\proof Let $U\subseteq \cA$ be an open set. We need to show that
$V=\{z\in \cA: \wh{F}(z)\subseteq U\}$ is open in $\cA$.  Let $z\in
V$, so that $\wh{F}(z)\subseteq U$. From Lemma \ref{l.filled_fibres}
it is easy to verify that $\wh{F}(z)$ is the intersection of all sets
of the form $[\wh{F}^-, \wh{F}^+]_{\cA}$ with $\wh{F}^-,
\wh{F}^+\in \cF_\cA$ and $\wh{F}^-\llcurly \wh{F}(z)\llcurly
\wh{F}^+$, and since this can be seen as a decreasing intersection,
such $\wh{F}^+$ and $\wh{F}^-$ may be chosen satisfying $[\wh{F}^-,
\wh{F}^+]_\cA\subset U$. Since $(\wh{F}^-, \wh{F}^+)_{\cA}$ is open
in $\cA$ and contains $z\in V$, which was chosen arbitrarily, it
follows that $V$ is open in $\cA$ as claimed.  \qed\medskip

\begin{lem} \label{l.projected_decomposition}
  $\cF_\cA$ projects to an upper semicontinuous decomposition $\cF_A=\{\pi(F)\mid
  F\in \cF_\cA\}$ of $A$ into cellular continua, and it has uncountably many elements.
\end{lem}
\proof Since by hypothesis there exists a self-homeomorphism $\varphi$
of $\A$ leaving $A$ invariant without periodic points in $A$, and we
may assume that $\varphi$ is orientation-preserving (replacing it by
$\varphi^2$ if necessary), Theorem~\ref{t.poincare} yields the existence
of a monotone map $h:A\to\kreis$, which lifts to a monotone map
$H:\cA\to \R$ that commutes with the translation $T$. It is easily
checked that for every $x\in \R$ the set $G_x=H^{-1}([x,x+1])$ is a
generator of $\cA$, and consequently every fibre
$H^{-1}(x)=T^{-1}(G_x)\cap G_x$ contains a cut. This further implies
that every fibre of $H$ contains (at least) one element of $\cF_\cA$.

In particular, the above yields that $T^k F\cap F=\emptyset$ for all
$F\in\cF_\cA$ and $k\in \Z\smin\{0\}$, hence $\cF_\cA$ projects to a
decomposition of $A$ with uncountably many elements, and each $F\in
\cF_\cA$ projects injectively into $A$. Since $F$ is a cellular
continuum, this implies that $\pi(F)$ is a cellular continuum as well.
Upper semicontinuity of $\cF_A$ then follows directly from that of
$\cF_\cA$ and from the fact that $\pi$ is an open map: if $U\subset A$ is a neighborhood of $F\in \cF_A$ in $A$, then the set of all elements of $\cF_A$ contained in $U$ is the projection of the set of elements of $\cF_\cA$ contained in the open set $\pi^{-1}(U) \subset \cA$, and therefore is open in $A$.
 \qed\medskip

Let $\cF$ be the decomposition of $A$ consisting of all elements of
$\cF_A$ together with all sets of the form $\{z\}$ with $z\notin A$.
Note that $\cF$ is a Moore decomposition of $\cA$, and therefore
Theorem \ref{t.moore} applies.

\begin{prop} \label{p.circle-image} The Moore projection $\Pi:\A\to\A$
  provided by Theorem \ref{t.moore} maps $A$ to an essential simple
  closed curve.
\end{prop}
\proof Due to Lemma \ref{l.curve-criterion} it suffices to show that
if $x,y \in \Pi(A)$ are different points, then $\Pi(A)\sm \{x,y\}$ is
disconnected.  By Lemma \ref{l.monotone_preimage}, this is the same as
saying that for any pair of distinct elements $F_1, F_2$ of $\cF_A$,
the set $A\sm (F_1\cup F_2)$ is disconnected. To show this, let
$\wh{F}_1, \wh{F}_2\in \cF_\cA$ be connected components of
$\Pi^{-1}({F}_1)$ and $\Pi^{-1}(F_2)$. Note that since
$T^{-1}F\llcurly F\llcurly TF$ for all $F\in \cF_\cA$, the projection
$\Pi$ is injective on $(\wh F_1, T\wh F_1)_\cA$, and there is
exactly one $k\in \Z$ such that $\wh F_1\llcurly T^k\wh F_2\llcurly
T\wh F_1$.  If $U=(\wh{F}_1, T^k\wh{F}_2)_{\cA}$ and
$V=(T^k\wh{F}_2,T\wh{F}_1)_{\cA}$, we have that $\Pi(U)$ and
$\Pi(V)$ are disjoint nonempty open subsets of $A$, and $\Pi(U)\cup
\Pi(V) = A\sm (F_1\cup F_2)$, proving that the latter set is
disconnected.  \qed\medskip

\begin{lem}\label{l.core}
  Suppose $h:A\to\kreis$ is a monotone map. Then the fibres $h^{-1}(x)$,
  $x\in\kreis$, are saturated with respect to $\cF_A$. 
\end{lem}
\proof The map $h$ lifts to a monotone map $H:\cA\to\R$ that commutes
with the deck translation $T$. As argued in the proof of
Lemma~\ref{l.projected_decomposition}, every fibre $H^{-1}(x)$
contains at least one cut. If $z\in H^{-1}(x)\cap \cB$ and $z'\in
H^{-1}(y)\cap \cB$ with $x<y$, this implies that $z\llcurly z'$ and
thus $F(z)\neq F(z')$. Hence, no two points of the same fibre of
$\cF_\cB$ can be contained in different fibres of $H$. In other words,
the fibres of $H_{|\cB}$ are saturated with respect to $\cF_\cB$. By
monotonicity, $F\ssq H^{-1}(x)$ implies $\wh F=F^\filled\ssq
H^{-1}(x)$, such that the fibres of $H$ are also saturated with
respect to $\cF_A$.  \qed\medskip

\subsection{Proof of Theorem~\ref{t.universal_factor_map}}

Denote by $\Pi:\A\to\A$ the Moore projection associated to
$\cF_\cA\cup\{x\in\A\mid x\notin A\}$, so it satisfies assertion (i) of Theorem~\ref{t.universal_factor_map}. Assertions (ii) and (iii)
follow from Proposition \ref{p.circle-image} and Sch\"onflies' theorem
(by post-composing with a homeomorphism of $\A$ which maps $\Pi(A)$ to
$\T$).  Since the lift $\Phi$ of
any homeomorphism $\varphi$ of $\A$ leaving $A$ invariant has to map
minimal generators to minimal generators, it follows from the above
constructions that $\Phi$ permutes the elements of $\cF_\cA$.  This
allows to define $\tilde\varphi:\A\to\A$ by requiring that
$\Pi^{-1}(\tilde \varphi(z))=\varphi(\Pi^{-1}(z))$. If $U\ssq \A$ is
open, then $\tilde\varphi^{-1}(U) = \Pi(\varphi^{-1}(\Pi^{-1}(U))$ is
open as well, since $\Pi$ maps saturated open sets to open sets.
Hence, $\tilde \varphi$ is continuous and assertion (iv) holds. 
To prove assertion (v), suppose that $h\colon A\to \kreis$ is
a monotone surjection, then by Lemma \ref{l.core}, $h^{-1}(x)$ is
saturated with respect to $\cF_A$, and thus $h$ induces a map
$\til{h}\colon \T\to \kreis$ by $\til{h}(x) = h(x')$ where $x'\in
\Pi^{-1}(x)$ is arbitrary. Note that $\til{h}$ is continuous since for
any open set $U\subset \kreis$ we have $\til{h}^{-1}(U) =
\Pi(h^{-1}(U))$ which is again open since it is the image of a saturated
open set. If $h$ is a semiconjugation from $\varphi$ to $R_\rho$, then
given $x\in \kreis$ and $x'\in \Pi^{-1}(x)$ we have
$$\tilde h\circ\tilde\varphi(x)=\tilde h\circ\tilde\varphi\circ\Pi(z) =\tilde
h\circ\Pi\circ\varphi(z)=h\circ\varphi(z)=R_\rho\circ h(z)=R_\rho\circ
\tilde h(x). $$
completing the proof if (v).
\qed

\subsection{Proof of Corollary \ref{c.uniqueness}}

Suppose $h_1,h_2$ are two semiconjugacies from $\varphi_{|A}$ to the
irrational rotation $R_\rho$, where $\rho$ is the rotation number of
$\varphi$ on $A$. By post-composing with a rotation, we may assume
that there exists $z_0\in A$ with $h_1(z_0)=h_2(z_0)$.

Let $\tilde\varphi$ be the homeomorphism of $\A$ from part (iv) of
Theorem~\ref{t.universal_factor_map}, so that $\Pi\circ
\varphi=\tilde\varphi\circ \Pi$ and $\tilde\varphi(\T)=\T$, and let
$\til{h}_i$ the maps such that $h_i = \til{h}\circ \Pi$ from part (v)
of the same theorem. Then $\til{h}_i$ semiconjugates
$\til{\varphi}|_{\T}$ to the irrational rotation $R_\rho$ of $\T$.
Since the semiconjugacies in the Poincar\'e Classification Theorem are
unique up to post-composition by a rotation, this yields that $\tilde
h_1=\tilde h_2$ and thus $h_1=h_2$. \qed

\section{Almost automorphic minimal continua}

A homeomorphism $f\colon X\to X$ of a metric space is called almost
automorphic if there exists $x\in X$ such that, whenever the limit
$\til{x}=\lim_{k\to \infty} f^{n_k}(x)$ exists for some sequence
$n_k\to \infty$, then $x=\lim_{k\to \infty} f^{-n_k}(\til{x})$.
Further, $f$ is almost periodic if for every $\epsilon>0$, the set
$\{n\in \Z : \forall x\in X,\, d(f^n(x),x)<\epsilon\}$ is syndetic
(\ie has uniformly bounded gaps).  The Veech Structure Theorem
\cite{veech1965almost} asserts that $f$ is almost automorphic if and
only if $f$ is semiconjugate to an almost periodic map of some space
$Y$ by means of an almost 1-1 continuous surjection, \ie a continuous
surjection $h\colon X\to Y$ for which the set of points of $Y$ with a
unique preimage is dense.

\subsection{Proof of Theorem \ref{t.almost_automorphy}}

Let $\varphi\colon \A\to \A$ be a homeomorphism, $B\subset \A$ is an
essential $\varphi$-invariant cobasin boundary without periodic
points, and $x_0\in B$ a recurrent bi-accessible point. Since
$x_1=\varphi(x_0)$ is also bi-accessible, we may find an inessential
simple closed curve $\gamma\subset \A$ intersecting $B$ exactly at
$x_0$ and $x_1$. If $V_0$ and $V_1$ are the two connected components
of $\A\sm \gamma$ and $B_i = V_i\cap B$, it is easy to verify that
$B_i\cup \{x_0, x_1\}$ is a continuum for $i\in\{0,1\}$. Thus $B$ is
decomposable, and Theorem \ref{t.universal_factor_map} implies that
there exists a $\Pi\colon \A \to \A$ mapping the circloid $A$ whose
boundary is $B$ to the circle $\T=\kreis\times \{0\}$ and inducing a
map $\til{\varphi}\colon \A\to \A$ which satisfies
$\tilde\varphi\circ\Pi=\Pi\circ\varphi$ and preserves $\T$. Let $U^-$
and $U^+$ denote the connected components of $\A\sm B$ which are
unbounded below and above, respectively, and let $C$ be a simple arc
joining a point of $U^-$ to a point of $U^+$ such that $C\cap
B=\{x_0\}$. Then $\til C = \Pi(C)$ is an arc joining a point below $\T$ to
a point above $\T$ and intersecting $\T$ exactly at
$\til{x}_0=\Pi(x_0)$. By continuity, the compact set
$\Pi^{-1}(\til{C})$ is contained in the closure of
$\Pi^{-1}(\til{C}\sm \{\til{x}_0\})$, and since $\Pi|_{U^-\cup U^+}$
is an injective map onto $\A\sm \T$ we have that $\Pi^{-1}(\til{C}\sm
\{\til{x}_0\})=C\sm \{x_0\}$. Thus $\Pi^{-1}(\til{C})\subset C$, and
$\Pi^{-1}(\til{x}_0) \subset C\cap B = \{x_0\}$.

Hence, $\Pi^{-1}(\Pi(x_0)) =\{x_0\}$, and the same is true for any
point in the orbit of $x_0$. Since $\til{\varphi}\colon \T\to \T$ is a
homeomorphism of the circle without periodic points, the classic
Poincar\'e theory implies that it is semi-conjugate to an irrational
rotation $R_\rho$ of the circle by means of a continuous monotone
surjection $h\colon \T\to \T^1$ which is injective in the nonwandering
set of $\til{\varphi}$. In particular, since $\til{x}_0$ is recurrent,
$h^{-1}(h(\til x_0)) = \{\til x_0\}$, and so $h' = h\circ \Pi$ is an
almost 1-1 semiconjugation between $\varphi|_B$ and $R_\rho$, showing
that $\varphi|_B$ is almost-automorphic.  \qed

\subsection{Proof of Corollary \ref{c.almost_automorphy}}

Let $U^-$ and $U^+$ be the connected components of $\A\sm X$ unbounded
below and above, respectively. Then $\bd U^- \cup \bd U^+$ is a
compact invariant subset of $X$ and therefore is equal to $X$. In
particular $\bd U^-\cap \bd U^+\neq \emptyset$, and since that is also
a compact invariant set we deduce $X=\bd U^-\cap \bd U^+$. Therefore,
$X$ is a cofrontier, and the Corollary follows from Theorem
\ref{t.almost_automorphy}.  \qed

\section{The Kuratowski decomposition} \label{Kuratowski}

Let $\Lambda$ be a continuum on the sphere $\mathbb{S}^2$ which is the
common boundary of two open simply connected sets. Using our
terminology, this is the same as saying that $\Lambda$ is a cobasin
boundary.

In \cite{kura-frontier}, Kuratowski studied the following topological
problem: when can we find a monotone surjection from $\Lambda$ onto
the circle? Following the method used in \cite{kura-II} to study a
similar question for irreducible continua, he defined a decomposition
into \emph{layers}\footnote{Translation of the original term
  \emph{tranches} (in french).} as follows: A fundamental layer of
$\Lambda$ is any subset of $\Lambda$ which is maximal with the
property of being a continuum which is the union of at most countably many subcontinua of
$\Lambda$, each of which either is indecomposable or has empty
interior in $\Lambda$.  When there is a unique fundamental layer, the
continuum $\Lambda$ is called \emph{monostratic}.

The main results of \cite{kura-frontier} implies that $\Lambda$ is
non-monostratic if and only if there exists some monotone continuous
surjection from $\Lambda$ onto the circle. In fact, if $\Lambda$ is
non-monostratic then the fundamental layers form a
monotone upper semicontinuous decomposition of $\Lambda$, and the
quotient space by this decomposition is a simple closed curve. In
addition, the decomposition into fundamental layers is the finest
upper semicontinuous monotone decomposition of $\Lambda$ into
subcontinua with the property of having the circle as a quotient
space, in the sense that any other such decomposition has its elements
saturated by fundamental layers.

This can be stated in terms of maps as follows:
\begin{thm} [Kuratowski]\label{th:kura}
  If $\Lambda$ is a non-monostratic cobasin boundary in
  $\mathbb{S}^2$, then there exists a monotone continuous surjection
  $\Pi\colon \Lambda\to \T^1$ such that for any other monotone
  continuous surjection $P\colon \Lambda \to \T^1$, there exists a map
  $\phi\colon \Lambda \to \Lambda$ such that $P = \phi\circ \Pi$.
\end{thm}

There are many works which extend to a more general setting the
concept of finding a finest monotone upper semicontinuous
decomposition with the property of having a quotient space with a
given property (what is often called the \emph{core} decomposition
with the given property); see for instance \cite{charatonik-1,
  fitzgerald-swingle, rakowski-1, vought-1} and references therein.

Using the previous theorem together with Moore's theorem, Kuratowski
obtains a result which can be reworded as follows \cite[Theorem
II]{kura-frontier}:

\begin{thm}\label{th:kura-2} If $\Lambda$ is a circloid in
  $\mathbb{S}^2$ with non-monostratic boundary, then there exists a
  monotone continuous surjection $\Pi\colon \mathbb{S}^2\to
  \mathbb{S}^2$ which maps $\Lambda$ onto the equator and is injective
  on $\mathbb{S}^2\sm \Lambda$.
\end{thm}
In fact, it can easily be verified that the decomposition of $\Lambda$
into fibers of $\Pi$ is the core decomposition with respect of having
the circle as quotient space. The map $\Pi$ can be characterized by
considering the decomposition of $\mathbb{S}^2$ into points of
$\mathbb{S}^2\sm \Lambda$ together with the ``filled'' layers of $\bd
\Lambda$ (\ie the union of each layer $L$ with all the connected
components of $\mathbb{S}^2\sm L$ which are disjoint from $\Lambda$).
This decomposition turns out to be the same given in the proof of
Theorem \ref{t.universal_factor_map}.  

In the proof of Theorem \ref{t.universal_factor_map}, we could have used Theorems~\ref{th:kura} and
\ref{th:kura-2}, but we chose a self-contained proof. We note that the only part of Section~\ref{UniversalFactorMap} where dynamics appears is to guarantee that there are uncountably many elements the decomposition $\cF_A$ (Lemma \ref{l.projected_decomposition}). Since it can be shown
without too much effort that uncountability of $\cF_\cA$ is equivalent to non-monostraticity, the results of Section~\ref{UniversalFactorMap} essentially contain an alternative proof of Kuratowski's results (if one assumes non-monostraticity instead of the dynamical hypothesis). 

In this context, the main contribution of Theorem
\ref{t.poincare} could also be stated as follows: if a circloid is invariant by a homeomorphism, has decomposable boundary and has no periodic points, then its boundary is non-monostratic.

\section{An example with large fibers}\label{AnosovKatokExample}

In this section we prove Theorem \ref{t.example}. More specifically,
we will prove that there exists a $C^\infty$ diffeomorphism $f\colon
\A\to \A$ and a monotone continuous surjection $\Pi\colon \A\to \A$
such that:
\begin{itemize}
\item $f$ leaves invariant an essential decomposable cofrontier $\Lambda$;
\item $f|_{\Lambda}$ is minimal;
\item $\Pi$ semi-conjugates $f$ with an irrational rotation of $\A$:
  $\Pi f = R_\rho \Pi$ for some $\rho\notin\Q$;
\item Each fiber of $\Pi|_{\Lambda}$ is a continuum with diameter at
  least $1/4$.
\end{itemize}

In what follows, intervals in $\T^1$ are assumed to be positively
oriented, so for $a,b\in \T^1$ the interval $(a,b)$ is the component
of $\T^1\sm \{a,b\}$ which is positively oriented and $(b,a)$ the
remaining one (and similarly for closed intervals).

\subsection{The Anosov-Katok method}

We will use the Anosov-Katok method, which we describe here briefly
(and refer to \cite{fayad/katok:2004} for more details and further
references).  The map $f$ will be obtained as a limit of maps $f_n$,
each of which is $C^\infty$-conjugate to a rational rotation
$R_{p_n/q_n}\colon (\theta, y)\mapsto (\theta+p_n/q_n, y)$, so
$$f_n = H_nR_{p_n/q_n} H_n^{-1}$$
where $H_n\in \diff^\infty(\A)$.
The maps $H_n$ are successive compositions of maps:
$$H_n = h_1 \circ \cdots \circ h_{n-1}\circ h_n.$$
The maps $h_n\in \diff^\infty(\A)$ are chosen inductively alongside
with the numbers $p_n/q_n$ with the following condition: If $h_n$ and
$p_n/q_n$ is already chosen, $h_{n+1}$ can be chosen arbitrarily, with
the only restriction that it commutes with $R_{p_n/q_n}$:
$$h_{n+1}R_{\frac{p_n}{q_n}} = R_{\frac{p_n}{q_n}}h_{n+1}.$$
This guarantees that for any $\tau$, 
$$H_{n+1}R_{\tau}H_{n+1}^{-1}= H_nh_{n+1}R_{\tau} h_{n+1}^{-1}H_n^{-1}
 = H_n\left(h_{n+1}R_{(\tau-\frac{p_n}{q_n})}h_{n+1}^{-1}\right)R_{\frac{p_n}{q_n}}H_{n}^{-1},$$
which means that if $\tau = p_{n+1}/q_{n+1}$ is chosen close enough to
$p_n/q_n$ (but different from it), the map $f_{n+1}$ can be made
arbitrarily $C^\infty$-close to $f_n$, and in particular one can make
the $C^n$-distance $\xi_n = d_{C^n}(f_{n+1}, f_n) +
d_{C^n}(f_{n+1}^{-1}, f_n^{-1})$ as small as desired.  Repeating
this process, if the numbers $\xi_n$ are chosen such that $\sum_{n\in
  \N} \xi_n<\infty$, the maps $f_n$ converge in the
$C^\infty$-topology to an element of $\diff^\infty(\A)$.

Since there is a great degree of freedom in the choice of $h_n$ as
well as $p_n/q_n$ at each step, additional restrictions may be placed
in order to guarantee that the limit map has special properties. To
outline the construction that we will follow, our choice of these maps
will be such that there exists a decreasing sequence of essential
annuli $A_n$ such that $h_{n+1}$ is the identity outside $A_n$ and
maps $A_{n+1}$ into $A_n$ in such a way that every horizontal circle
in $A_{n+1}$ becomes $\epsilon$-dense in $A_n$ (with an appropriate
choice of $\epsilon$ depending on $n$) and every vertical segment in
$A_{n+1}$ is mapped to a set with diameter greater than $1/4$. This is
achieved by first using a map that ``twists'' vertical segments inside
$A_{n+1}$ so that any such segment becomes horizontally large, and
then composing with a map that maps $A_{n+1}$ to something that
oscillates vertically inside $A_n$ (see Figure \ref{fig:example}).
\begin{figure}[ht]
\begin{center}
\includegraphics[width=.9\linewidth]{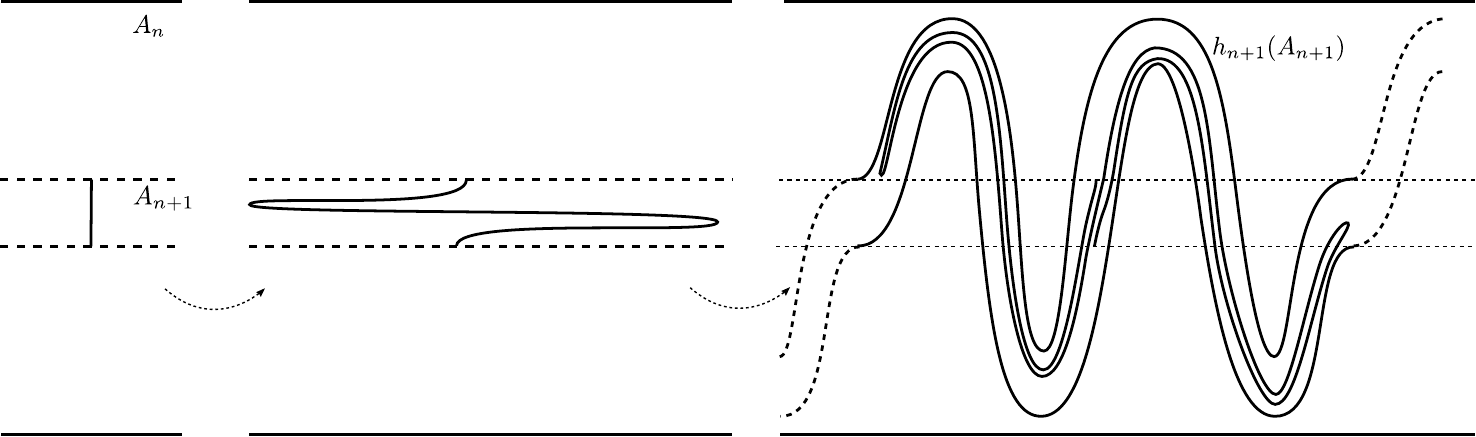}
\caption{How $h_{n+1}$ maps a vertical segment in $A_{n+1}$ (in two steps).}
\label{fig:example}
\end{center}
\end{figure} 

The cofrontier $\Lambda$ will be the intersection of the sets
$\Lambda_n = H_n(A_n)$, which is a decreasing intersection. An
appropriate choice of $p_{n+1}/q_{n+1}$ and the fact that $h_{n+1}$
spreads $\epsilon$-densely in $A_n$ every horizontal circle will
guarantee that $f|_\Lambda$ is minimal and $p_n/q_n$ converges to an
irrational number $\alpha$. Moreover, although the maps $H_n$ are not
required to converge to a homeomorphism, we will guarantee that
$H_n^{-1}$ does converge to a continuous surjection $\Pi$ in the $C^0$
topology. Due to the way in which these maps are defined, this
automatically implies that $\Pi$ semi-conjugates $f$ with the rotation
$R_\alpha$, and the conditions on $h_n$ will guarantee that the
preimage of any point of $\Lambda$ by $\Pi$ has diameter at least
$1/4$.\medskip

Let $(\alpha_n)_{n\in \N}$ and $(\epsilon_n)_{n\in\N}$ be decreasing
sequences of positive real numbers such that $\alpha_1<1$,
$\alpha_n\to 0$ as $n\to \infty$, and $\sum_{n\in \N} \epsilon_n
<\infty$.

\begin{claim}\label{claim:ak-1}
  There exist sequences $(h_n)_{n\in \N}$ in $\diff^\infty(\A)$,
  $(p_n/q_n)_{n\in \N}$ in $\Q$ (where $p_n, q_n$ are relatively prime
  integers, $q_n>0$), and $(M_n)_{n\in \N}$ in $\N$ such that, letting
\begin{itemize}
\item $H_n = h_1h_2\cdots h_n$,
\item $f_n = H_nR_{p_n/q_n}H_n^{-1}$,
\item $A_n = \T^1\times [-\alpha_n, \alpha_n]$,
\item $\Lambda_n = H_n(A_n)$,
\end{itemize}
the following properties hold for $k\geq 1$:
\begin{enumerate}[label={(\arabic*)}]
\item\label{it:1} $h_{k}(A_k) = A_k$, and $h_{k}(z)=z$ for all $z\in \A\sm A_k$;
\item\label{it:4} $\diam(H_k(\{\theta\}\times [-\alpha_k, \alpha_k]))
  > 1/4$ for all $\theta\in \T^1$;
\item\label{it:8} $q_k>13^k$;
\end{enumerate}
and for $k\geq 2$,
\begin{enumerate}[label={(\arabic*)}]
\setcounter{enumi}{3}
\item\label{it:2} $h_{k}R_{p_{k-1}/q_{k-1}} = R_{p_{k-1}/q_{k-1}}h_{k}$;
\item\label{it:3} $H_{k}(\T^1\times \{\alpha\})$ is
  $\epsilon_{k-1}$-dense in $\Lambda_{k-1}$ for all $\alpha\in
  [-\alpha_{k}, \alpha_{k}]$;
\item\label{it:5} $d(\pi_1(h_{k}(z)), \pi_1(z))<3/q_{k-1}$ for all $z\in\A$;
\item\label{it:6} $d_{C^{k}}(f_{k},f_{k-1})<\epsilon_{k-1}$ and
  $d_{C^{k}}(f_{k}^{-1},f_{k-1}^{-1})<\epsilon_{k-1}$;
\item\label{it:6b} $|p_{k}/q_{k} - p_{k-1}/q_{k-1}|<\epsilon_{k-1}$;
\item\label{it:7} $\{f^j_{k}(z) : 0\leq j\leq M_{i+1}\}$ is
  $\epsilon_{i}$-dense in $\Lambda_i$ for all $z\in \Lambda_{k}$ and
  $1\leq i\leq k-1$;
\end{enumerate}
\end{claim}

Before explaining how to obtain such sequences, let us prove that they
lead to a map with the required properties.

\begin{claim}\label{claim:ak-2} Using the maps defined in Claim
  \ref{claim:ak-1}, there exist $f\in \diff^\infty(\A)$, $\Pi\in
  C^0(\A,\A)$, and $\rho\in \T^1$ such that, if $\Lambda =
  \bigcap_{n\in \N}\Lambda_n$, then
\begin{itemize}
\item $f_n\to f$ in the $C^\infty$ topology as $n\to \infty$;
\item $H_n^{-1}\to \Pi$ in the $C^0$ topology as $n\to \infty$;
\item $\Pi$ is a monotone semiconjugation between $f$ and $R_\rho$;
\item $\Lambda$ is an essential $f$-invariant decomposable cofrontier
  and $\Pi(\Lambda)=\T^1\times \{0\}$;
\item $\diam(\Pi^{-1}(\Pi(z))) \geq 1/4$ for all $z\in \Lambda$, and
  $\Pi|_{\A\sm \Lambda}$ is injective;
\item $f|_{\Lambda}$ is minimal.
\end{itemize}
\end{claim}

\begin{proof}
Note that
$$d_{C^0}(H^{-1}_n, H^{-1}_{n+m}) = d_{C^0}(H^{-1}_qnH_{n+m}, 
H^{-1}_{n+m}H_{n+m}) = d_{C^0}(h_{n+1}h_{n+2}\cdots h_{n+m}, \id).$$
By \ref{it:5} one has
\[
|\pi_1(h_{n+1}h_{n+2}\cdots h_{n+m}(z))-\pi_1(z)| \leq \sum_{i=n}^{n+m-1} \frac{3}{q_{i}},
\]
and by \ref{it:1}, since each $h_{n+i}$ leaves $A_n$ invariant and
is the identity outside $A_n$,
\[
|\pi_2h_{n+1}h_{n+2}\cdots h_{n+m}(z)-\pi_2(z)| \leq 2\alpha_n.
\]
Since $\alpha_n\to 0$ and $\sum_i 1/q_i < \infty$, we see that the two
coordinates of the map $$z\mapsto h_{n+1}h_{n+2}\cdots h_{n+m}(z)-z$$ are
uniformly small if $n,m$ are large enough, and therefore
$(H^{-1}_i)_{i\in \N}$ is a Cauchy sequence in $C^0(\A, \A)$, so it
converges to some continuous (and surjective) map $\Pi$. Note also
that $H^{-1}_i(z)$ is eventually constant if $z\notin \Lambda$. Since
each $H^{-1}_i$ is a homeomorphism onto its image, it follows easily
that $\Pi|_{\A\sm \Lambda}$ is injective, and
$\Pi(\Lambda)=\T^1\times\{0\}$.

By \ref{it:6} and the fact that $\sum_n {\epsilon_n}<\infty$, the
sequences $(f_n)_{n\in \N}$ and $(f_n^{-1})_{n\in\N}$ are Cauchy in
the $C^r$-topology for any $r$, the sequence $(f_n)_{n\in \N}$
converges to some $C^\infty$ diffeomorphism $f$ in the
$C^\infty$-topology.

By \ref{it:6b}, the sequence $(p_n/q_n)_{n\in \N}$ is Cauchy and
therefore has some limit $\rho\in \T^1$.
oreover, since
$$H^{-1}_n f_n = H_n^{-1}(H_nR_{p_n/q_n}H_n^{-1}) = R_{p_n/q_n} H_{n}^{-1},$$
taking limits as $n\to \infty$ one deduces that $\Pi f = R_{\rho}\Pi$.
The fact that $\Pi$ is a uniform limit of homeomorphisms implies that
$\Pi$ is monotone (see for instance \cite{whyburn-analytic}), so the preimages of points are connected.
Moreover, note that $H_n(\{\theta\}\times[-\alpha_n,\alpha_n])$ is
connected and has diameter at least $1/4$ due to \ref{it:4}.
Choosing $z_n, z_n'\in H_n(\{\theta\}\times[-\alpha_n,\alpha_n])$ such
that $d(z_n,z_n')\geq 1/4$ and taking convergent subsequences of
$(z_n)$ and $(z_n')$ one finds two points $z,z'\in \Lambda$ such that
$d(z,z')\geq 1/4$ and $\Pi(z)=\Pi(z')=(\theta,0)$, showing that
$\diam(\Pi^{-1}(\theta,0)) \geq 1/4$. This means that
$\diam(\Pi^{-1}(\Pi(z)))\geq 1/4$ for all $z\in \Lambda$.

The minimality of $f|_{\Lambda}$ follows immediately from
\ref{it:7}.  Since the set $\Lambda= \bigcap_{n\in \N} \Lambda_n$ is
a decreasing intersection of essential closed topological annuli, it
is an essential annular continuum, and using the fact that
$f|_\Lambda$ is minimal we deduce that $\Lambda$ is an essential
cofrontier (by an argument already used in the proof of Corollary
\ref{c.almost_automorphy}).

To see that $\Lambda$ is decomposable, let $I_1 = [0,1/2]$ and $I_2 =
[1/2,1]$ be the upper and lower half-circles, respectively, and let
$\Lambda_n^j = H_n(I_j\times [-\alpha_n, \alpha_n])$ for $j\in
\{1,2\}$. For each $j$, let $\Lambda^j$ be the Hausdorff limit of some
convergent subsequence of $(\Lambda_n^j)_{n\in \N}$. Then $\Lambda^1$
and $\Lambda^2$ are compact connected nonempty sets, and from the fact
that $\Lambda_n=\Lambda^1_n\cup \Lambda^2_n$ it is easy to verify that
$\Lambda = \Lambda^1\cup \Lambda^2$. On the other hand, note that from
\ref{it:8} one has $\mu := \sum_{i=1}^\infty 3/q_i <
3\sum_{i=1}^\infty 1/13^i = 1/4$, and \ref{it:5} implies that
$d(\pi_1(H_n(z)), \pi_1(z))<\mu$ (as noted in the beginning of the
proof). Thus $\pi_1(\Lambda^j_n)=
\pi_1(H_n(I_j\times[-\alpha_n,\alpha_n]))$ is an interval of length at
most $1/2 + 2\mu$. Therefore $\pi_1(\Lambda^j)$ has length at most
$1/2+2\mu < 1$, and since $\Lambda$ is essential this means that
$\Lambda^j\neq \Lambda$ for $j\in \{1,2\}$, showing that $\Lambda =
\Lambda^1\cup \Lambda^2$ is a decomposition of $\Lambda$ into nonempty
proper subcontiua.
\end{proof}

\begin{proof}[Proof of Claim \ref{claim:ak-1}]
  We define the sequences recursively. Start with $p_1=0$ and
  $q_1=13^2$. To simplify the construction we will assume
  $\alpha_1>1/4$, so that letting $h_1=\id_\A$ conditions
  \ref{it:1}-\ref{it:8} automatically hold for $k=1$. Note that (9) is vacuously true when $k=1$, so we choose $M_1$ arbitrarily.

  Assuming that $h_k$, $p_k/q_k$, and $M_k$ are already defined for
  $1\leq k\leq n$ and satisfy the required properties, we will define
  them for $k=n+1$.
  As explained in the introduction of this section, as long as
  $p_{n+1}/q_{n+1}$ is chosen close enough to $p_n/q_n$ we can
  guarantee that $f_{n+1}$ will be arbitrarily close to $f_{n}$ in the
  $C^n$-topology.  Therefore we will first choose $h_{n+1}$
  accommodating to our needs, and later we will choose
  $p_{n+1}/q_{n+1}$ close enough to $p_n/q_n$ (but different from it)
  so as to guarantee conditions \ref{it:6}, \ref{it:7} and
  \ref{it:8}.

  In order to define our map $h_{n+1}$ we first introduce two
  auxiliary maps. Given $\delta>0$ and positive integers $K,n,q$ the
  first map $V=V_{n,q,\delta}\in \diff^\infty(\A)$ will have the
  following properties:
\begin{itemize}
\item[(V1)] $V(z) = z$ for $z\in \A\sm A_n$;
\item[(V2)] If $I\subset \T^1$ is any interval of length greater than
  $\delta$, then there is a subinterval of $I$ with endpoints
  $\theta_0, \theta_1$ such that $V$ maps each set
  $\{\theta_0\}\times[-\alpha_{n+1},\alpha_{n+1}]$ and
  $\{\theta_1\}\times[-\alpha_{n+1},\alpha_{n+1}]$ into a different
  connected component of $A_n\sm (\T^1\times [-\alpha_n+\delta,
  \alpha_n-\delta])$.
\item[(V3)] $d(\pi_1(V(z)), \pi_1(z)) < \delta$ for all $z\in \A$.
\item[(V4)] $V$ commutes with $R_{\frac{1}{q}}$.
\end{itemize}

There are several ways to define such a map. We describe one
way\footnote{An alternative definition (which we will not develop
  further) leads to an area-preserving map $V$ with similar
  properties; roughly, this is done by using a map which expands
  $A_{n+1}$ horizontally, contracts vertically, and folds it back into
  $A_n$ while keeping all the required properties. By doing this, one
  could guarantee that the maps $f_n$ (and therefore limit map $f$ in
  Claim \ref{claim:ak-2}) are area-preserving.}, depicted in Figure
\ref{fig:map-V}: let $\phi\colon \R\to \R$ be a $C^\infty$ bump
function such that $\phi(t)=1$ if $|t|<\alpha_{n+1}$, $\phi(t)=0$ if
$|t|>\alpha_n$ and $0\leq \phi(t)\leq 1$ for all $t\in \R$, and
consider the vector field
$$V(\theta, y)=\left(0,\, \sin\big (2\pi \frac{\delta x}{4q}\big)\phi(y)\right)$$ on $\A$. 
If $t\in \R$ is large enough, the time-$t$ map of the flow induced by $V$ is a 
diffeomorphism satisfying the requred properties (see Figure \ref{fig:map-V}).

\begin{figure}
\begin{center}
\includegraphics[height=5cm]{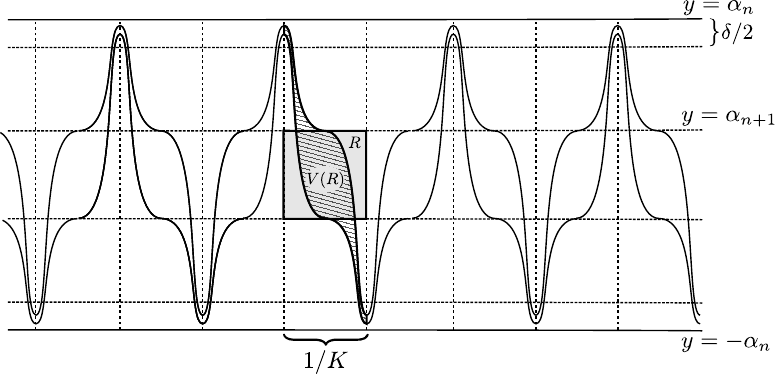}
\caption{A possible map $(\theta,y)\mapsto V(\theta,y)$}
\label{fig:map-V}
\end{center}
\end{figure}

The second map $T=T_{n,q}\in \diff^\infty(\A)$ is such that:
\begin{itemize}
\item[(T1)] $T(z)=z$ for all $z\in \A\sm A_n$ and all $z\in \T^1\times \{0\}$;
\item[(T2)] $T(A_{n+1})=A_{n+1}$;
\item[(T3)] $[\theta-1/q, \theta+1/q]\subset
  \pi_1(T(\{\theta\}\times[-\alpha_{n+1}, \alpha_{n+1}]))$ for all
  $\theta\in \T^1$;
\item[(T4)] $d(\pi_1(T(z)), \pi_1(z)))\leq 2/q$ for all $z\in \A$;
\item[(T5)] $T$ commutes with $R_{1/q}$.
\end{itemize}

Such a map may be defined explicitly as 
$$(\theta,y)\mapsto \left(\theta+\frac{\psi(y)y}{\alpha_{n+1}q_n},\, y\right),$$
where $\psi\colon \R\to \R$ is a $C^\infty$ bump function such that
$0\leq \psi(z)\leq 1$ for all $z\in \R$, $\psi$ is equal to $1$ on
$[-\alpha_{n+1}, \alpha_{n+1}]$ and $0$ outside $[-\alpha_n,
\alpha_n]$ and $|\psi(y)y|\leq 2\alpha_{n+1}$ for all $y\in\R$.

Let $s= \sup_{\theta\in \T^1} \diam(H_n(\{\theta\}\times [-\alpha_n,
\alpha_n])) - 1/4$. By continuity of $H_n$ and property \ref{it:4} it
follows that $s>0$. Let $\epsilon = \min\{s, \epsilon_{n+1}\}$, and
choose $0<\delta$ so small that $d(H_n(z), H_n(z'))<\epsilon/2$
whenever $d(z,z')<3\delta$, and in addition $\delta<1/q_n$ and
$\delta<(\alpha_{n}-\alpha_{n+1})/4$.

We now claim that the map $h_{n+1} = V\circ T$ satisfies the required
properties using $q=q_n$. First, note that properties \ref{it:1} and
\ref{it:2} for $k=n+1$ follow directly from (V1), (T1), (V4) and (T5).

Let us further note that whenever $\gamma\subset A_{n+1}$ is an arc such that
$\diam(\pi_1(\gamma))>\delta$, then $V(\gamma)$ is $3\delta$-dense in
$\pi_1(\gamma)\times [-\alpha_n, \alpha_n]$. Indeed, the properties of
$V$ imply that $[-\alpha_{n}+\delta, \alpha_n-\delta]\subset
\pi_2(V(\gamma))$ and $\pi_1(V(\gamma))$ lies in an
$\delta$-neighborhood of $\pi(\gamma)$. Thus, given $z\in
\pi_1(\gamma)\times [-\alpha_n+\delta, \alpha_n-\delta]$, there exists
$z'\in \gamma$ such that $\pi_2(z')=\pi_2(z)$ and $d(z', z)<2\delta$.
From this it follows easily that $V(\gamma)$ is $3\delta$-dense in
$\pi_1(\gamma)\times [-\alpha_n, \alpha_n]$ as claimed.

Since, for any $\alpha\in [-\alpha_{n+1}, \alpha_{n+1}]$, the set
$T(\T^1\times\{\alpha\})$ is contained in $A_{n+1}$, the previous
remarks imply that
$h_{n+1}(\T^1\times\{\alpha\})=V(T(\T^1\times\{\alpha\}))$ is
$3\delta$-dense in $A_n$, anqd since
$H_{n+1}(\T^1\times\{\alpha\})=H_n(h_{n+1}(\T^1\times \{\alpha\}))$,
the choice of $\delta$ implies that $H_{n+1}(\T^1\times\{\alpha\})$ is
$\epsilon$-dense in $H_n(A_n)=\Lambda_n$. Since
$\epsilon<\epsilon_{n+1}$, this shows that property \ref{it:3} holds
for $k=n+1$.
 
Given $\theta\in \T^1$, let $\gamma=T(\{\theta\}\times [-\alpha_{n+1},
\alpha_{n+1}])$. Then $[\theta-1/q_n,\theta+1/q_n]\subset
\pi_1(\gamma)$, and in particular $\diam(\pi_1(\gamma))> 2/q_n>
\delta$, so that $V(\gamma)$ is $3\delta$-dense in
$\pi_1(\gamma)\times [-\alpha_n,\alpha_n]$. Recall from our choice of
$\epsilon$ that the diameter of $H_n(\{\theta\}\times
[-\alpha_n,\alpha_n])$ is at least $1/4+\epsilon$. Thus we can find
$z_1,z_2\in \{\theta\}\times [-\alpha_n,\alpha_n]$ such that
$d(H_n(z_1), H_n(z_2))\geq 1/4+\epsilon$. But then there exist
$z_1',z_2'\in V(\gamma)$ such that $d(z_i,z_i') < 3\delta$, and our
choice of $\delta$ implies that $d(H_n(z_i), H_n(z_i'))<\epsilon/2$
for $i\in \{1,2\}$. Thus $d(H_n(z_1'), H_n(z_2'))>1/4$, and it follows
that $H_n(V(\gamma))$ has diameter greater than $1/4$. Since
$V(\gamma)=h_{n+1}(\{\theta\}\times [-\alpha_{n+1}, \alpha_{n+1}])$,
we see that property (\ref{it:4}) holds for $k=n+1$.

To verify property (\ref{it:5}) for $k=n+1$, note that, using the
properties of $V$ and $T$ and the fact that $\delta<1/q_n$,
$$d(\pi_1(h_{n+1}(z), \pi(z)) = d(\pi_1(V(T(z))), \pi_1(z)) \leq  \delta + 2/q_n < 3/q_n.$$

The remaining properties are guaranteed by choosing $p_{n+1}/q_{n+1}$
close enough to $p_n/q_n$. Indeed, as noted at the beginning of the
proof, properties \ref{it:6} and \ref{it:7} are guaranteed in this
way, and property \ref{it:8} will hold if $p_{n+1}/q_{n+1}$ is close
to $p_n/q_n$ but not equal to it. Therefore there is an open interval
$I$ containing $p_n/q_n$ such that properties \ref{it:6}
\ref{it:6b} and \ref{it:8} hold as long as $p_{n+1}/q_{n+1}\in
I\sm\{p_n/q_n\}$.

By our induction assumption, property \ref{it:7} holds for $f_n$, so
that $\{f^k_n(z): 0\leq k\leq M_i\}$ is $\epsilon_i$-dense in
$\Lambda_i$ for each $z\in \Lambda_n$ and $1\leq i\leq n-1$. However,
this condition is $C^0$-open, meaning that it still holds if one
replaces $f_n$ by a map which is $C^0$-close enough to $f_n$.
Therefore, reducing the interval $I$ we may assume that whenever
$p_{n+1}/q_{n+1}\in I\sm\{p_n/q_n\}$, the set $\{f^k_{n+1}(z) : 0\leq
k\leq M_{i+1}\}$ is $\epsilon_i$-dense in $\Lambda_i$ when $1\leq
i\leq n-1$ (noting that $\Lambda_{n+1}\subset \Lambda_n$). It remains
to verify that there exists $M_{n+1}$ such that the latter also holds
when $i=n+1$.

Choose an irrational $\theta\in I$ and note that by property
\ref{it:3} the set $H_{n+1}(\T^1\times \{\alpha\})$ is
$\epsilon_{n+1}$-dense in $\Lambda_n$ for any $\alpha\in
[-\alpha_{n+1},\alpha_{n+1}]$, so the map
$g=H_{n+1}R_{\theta}H_{n+1}^{-1}$ has the property that every
$g$-orbit of a point in $\Lambda_{n+1}$ is $\epsilon_{n+1}$-dense in
$\Lambda_n$. By compactness, we may choose $M_{n+1}$ such that
$\{g^k(z):0\leq k\leq M_{n+1}\}$ is $\epsilon_{n+1}$-dense in
$\Lambda_n$ for any $z\in \Lambda_{n+1}$. Since this is an open
condition, if $p_{n+1}/q_{n+1}\in I\sm\{p_n/q_n\}$ is chosen close
enough to $\theta$ we have that $\{f_{n+1}^k(z):0\leq k\leq M_{n+1}\}$
is $\epsilon_{n+1}$-dense in $\Lambda_n$, as required. Thus property
\ref{it:7} holds, completing the proof.

\end{proof}

%\footnotesize
\providecommand{\bysame}{\leavevmode\hbox to3em{\hrulefill}\thinspace}
\providecommand{\MR}{\relax\ifhmode\unskip\space\fi MR }
\providecommand{\MRhref}[2]{%
  \href{http://www.ams.org/mathscinet-getitem?mr=#1}{#2}
}
\providecommand{\href}[2]{#2}

\end{document}